\documentclass[12pt, a4paper]{amsart}
\usepackage{amsthm,amsmath,amssymb,times}
\usepackage{graphicx,fullpage,xcolor}
\usepackage[colorlinks,linkcolor=blue]{hyperref}
\usepackage{threeparttable,booktabs}
\usepackage{mathrsfs}
\usepackage{subfigure,stmaryrd}
\usepackage{cite,makecell}
\usepackage{url,multicol,mathtools}
\usepackage{appendix,longtable}
\hypersetup{plainpages=True, pdfstartview=FitH, bookmarksopen=true,
pdfpagemode=none, colorlinks=true,linkcolor=blue,citecolor=blue}
\usepackage{shadow}
\usepackage{geometry}
\usepackage[utf8]{inputenc}
\usepackage{tikz}
\usetikzlibrary{calc,decorations.markings}
\usepackage{comment}
\usepackage{CJKutf8}
\usepackage[textfont=it]{caption}
\usepackage{ifthen}

\def\le{\leqslant}
\def\ge{\geqslant}
\def\tr#1{\lfloor #1\rfloor}

\def\cl#1{\lceil #1\rceil}

\def\lpa#1{\bigl({#1}\bigr)}
\def\Lpa#1{\Bigl({#1}\Bigr)}

\def\llpa#1{\biggl({#1}\biggr)}
\def\LLT{\mathrm{LLT}}
\def\simeqq{\,\simeq\,}
\def\dd#1{{\,\rm d}#1}
\def\ve{\varepsilon}
\def\tr#1{\left\lfloor #1\right\rfloor}

\def\thick{0.25}
\def\maxHeight{5}
\def\mainColor{black}

\def\verofcha{tiny}
\def\ifsize{\ifthenelse{\equal{\verofcha}{tiny}}{\tiny}{\footnotesize}}
\newcommand{\test}[4]{
    \minipage[t]{1.8cm}
        \centering
        \begin{tikzpicture}[scale=0.2]
            \foreach \y [count=\x] in #1 {
                \coordinate (A\x) at ($( \x, \maxHeight - \y)$);
                \filldraw[\mainColor] ($(\x-\thick, 0)$) -- ($(\x+\thick, 0)$) --($(A\x) + (\thick, \thick)$) -- ($(A\x) + (-\thick, \thick)$) -- cycle;
            };
            \foreach \i/\j in #2 {
                \filldraw[\mainColor] ($(A\i) + (\thick, \thick)$) -- ($(A\i) + (\thick, -\thick)$) -- ($(A\j) + (-\thick, -\thick)$) -- ($(A\j) + (-\thick, \thick)$) -- cycle;
            };
        \end{tikzpicture}\\
        {#3}\\
        {#4}
    \endminipage\vspace{1pt}
}

\DeclareRobustCommand{\Stirling}{\genfrac\{\}{0pt}{}}
\DeclareRobustCommand{\stirling}{\genfrac[]{0pt}{}}

\makeatletter
\@ifpackageloaded{MnSymbol}{}{
  \DeclareFontFamily{OMX}{MnSymbolE}{}
  \DeclareSymbolFont{largesymbolsMn}{OMX}{MnSymbolE}{m}{n}
  \SetSymbolFont{largesymbolsMn}{bold}{OMX}{MnSymbolE}{b}{n}
  \DeclareFontShape{OMX}{MnSymbolE}{m}{n}{
      <-6>  MnSymbolE5
     <6-7>  MnSymbolE6
     <7-8>  MnSymbolE7
     <8-9>  MnSymbolE8
     <9-10> MnSymbolE9
    <10-12> MnSymbolE10
    <12->   MnSymbolE12}{}
  \DeclareFontShape{OMX}{MnSymbolE}{b}{n}{
      <-6>  MnSymbolE-Bold5
     <6-7>  MnSymbolE-Bold6
     <7-8>  MnSymbolE-Bold7
     <8-9>  MnSymbolE-Bold8
     <9-10> MnSymbolE-Bold9
    <10-12> MnSymbolE-Bold10
    <12->   MnSymbolE-Bold12}{}
  \DeclareMathDelimiter{\llangle}{\mathopen}{largesymbolsMn}{'164}
                                            {largesymbolsMn}{'164}
  \DeclareMathDelimiter{\rrangle}{\mathclose}{largesymbolsMn}{'171}
                                            {largesymbolsMn}{'171}
}
\makeatother

\newtheorem{thm}{Theorem}

\newtheorem{lmm}{Lemma}
\newtheorem{cor}{Corollary}
\newtheorem{prop}{Proposition}

\graphicspath{{./MA-figs/}}

\def\and{\mbox{\quad and \quad}}

\begin{document}
\begin{CJK}{UTF8}{ipxg}
\title[Bell numbers in Matsunaga's and Arima's
Genjik\={o} combinatorics]{Bell numbers in Matsunaga's and Arima's
Genjik\={o} combinatorics: \\ Modern perspectives and local limit theorems}

\let\origmaketitle\maketitle
\def\maketitle{
  \begingroup
  \def\uppercasenonmath##1{}
  \let\MakeUppercase\relax
  \origmaketitle
  \endgroup
}

\author{Xiaoling Dou, Hsien-Kuei Hwang, Chong-Yi Li}

\thanks{The research of the third author was partially supported
by an Investigator Award from Academia Sinica under the Grant
AS-IA-104-M03 and by Taiwan Ministry of Science and Technology under
the Grant MOST 108-2118-M-001-005-MY3.}

\address{Center for Data Science, Waseda University, 1-6-1
Nishiwaseda, Shinjuku-ku, Tokyo 169-8051, Japan}
\email{\textbf{dou@waseda.jp}}

\address{Institute of Statistical Science, Academia Sinica, Taipei,
115, Taiwan}
\email{\textbf{hkhwang@stat.sinica.edu.tw}}
\email{\textbf{chongyili@stat.sinica.edu.tw}}

\maketitle

\begin{abstract}
We examine and clarify in detail the contributions of Yoshisuke
Matsunaga (1694?--1744) to the computation of Bell numbers in the
eighteenth century (in the Edo period), providing modern perspectives to some unknown materials that are by far the earliest in the history of Bell numbers. Later clarification and developments by Yoriyuki Arima (1714--1783), and several new results such as the asymptotic distributions (notably the corresponding local limit theorems) of a few closely related sequences are also given.
\end{abstract}
\section{Introduction}

\subsection{Bell numbers}
The Bell numbers $B_n$, counting the total number of ways to
partition a set of $n$ labeled elements, are so named by Becker and
Riordan \cite{Becker1948}. While it is known that their first
appearance can be traced back to the Edo period in Japan (see
\cite{Gould1979,Knuth2011,Luschny2011}), the history of these numbers
is ``\emph{a tricky business}'' \cite[p.~105]{Pollard2003}, and
``\emph{the earliest occurrence in print of these numbers has never
been traced}'' \cite{Rota1964}; Rota added in \cite{Rota1964}:
``\emph{as expected, the numbers have been attributed to Euler, but
an explicit reference has not been given}'', obscuring further the
early history of Bell numbers. Furthermore, in the words of Bell
\cite{Bell1938}: the $B_n$ ``\emph{have been frequently investigated;
their simpler properties have been rediscovered many times}''. The
recurrent rediscoveries, as well as the wide occurrence in diverse
areas, in the last three centuries certainly testify the importance
and usefulness of Bell numbers. See \cite{Mendelsohn1951} for an
instance of a typical rediscovery, and the OEIS webpage of Bell
numbers \href{https://oeis.org/A000110}{OEIS A000110} for the diverse
contexts where they appear. In particular, Bell numbers rank the 26th
(among a total of 347,900+ sequences as of September 21, 2021) according to
the number of referenced sequences in the OEIS database.

Bell numbers can be characterized and computed in many ways;
indeed a few dozens of different expressions for characterizing $B_n$
are available on the OEIS webpage
\cite[\href{https://oeis.org/A000110}{A000110}]{oeis2021}. Among
these, one of the most commonly used that is also by far the earliest one,
often attributed to Yoshisuke Matsunaga's unpublished work in the
eighteenth century (see, e.g., \cite[p.~504]{Knuth2011} or
\cite[Theorem 1.12]{Mansour2016}), is the recurrence
\begin{align}\label{E:bell-formulae}
    B_n &= \sum_{0\le k<n}\binom{n-1}{k}B_{n-1-k} ,
\end{align}
for $n\ge1$ with $B_0=1$. A more precise reference of this recurrence
is Arima's book ``\emph{Sh\={u}ki Sanp\={o}}'' (\emph{Collections in
Arithmetics} \cite{Arima1769}). More precisely, Knuth writes (his
$\varpi_n$ is our $B_n$):
\begin{quote}\small
    ``Early in the 1700s, Takakazu Seki and his students began to
    investigate the number of set partitions $\varpi_n$ for
    arbitrary $n$, inspired by the known result that
    $\varpi_5=52$. Yoshisuke Matsunaga found formulas for the
    number of set partitions when there are $k_j$ subsets of
    size $n_j$ for $1\le j\le t$, with $k_1n_1+\cdots+k_tn_t=n$
    (see the answer to exercise 1.2.5--21). He also discovered
    the basic recurrence relation 7.2.1.5--(14), namely
    \[
         \varpi_{n+1}
         = \binom{n}{0}\varpi_n + \binom{n}{1}\varpi_{n-1} +
         \binom{n}{2}\varpi_{n-2} +\cdots+\binom{n}{n}\varpi_0,
    \]
    by which the values of $\varpi_n$ can readily be computed.

    Matsunaga's discoveries remained unpublished until Yoriyuki
    Arima's book \emph{Sh\=uki Sanp\=o}\/ came out in 1769.
    Problem $56$ of that book asked the reader to solve the
    equation ``$\varpi_n=678570$'' for $n$; and Arima's answer,
    worked out in detail (with credit duly given to Matsunaga),
    was $n=11$.''
\end{quote}

\hyphenation{Higashi-oka}

\subsection{Yoshisuke Matsunaga}

However, as will be clarified in this paper, Yoshisuke Matsunaga
({\footnotesize{松永良弼}})\footnote{For the reader's convenience,
Kanji characters will be added at their first occurrence whenever
possible in what follows because the correspondence between Japanese
romanization and the Kanji character is often not unique.} indeed
used a very different procedure (see Theorem~\ref{T:Matsunaga} below)
in his 1726 book \cite{Matsunaga1726} to compute $B_n$, which was
later expounded in detail and modified by Yoriyuki Arima
({\footnotesize{有馬頼徸}}) in his 1763 book \cite{Arima1763} (not
his 1769 \emph{Sh\={u}ki Sanp\={o}} \cite{Arima1769}), eventually led
to the recurrence \eqref{E:bell-formulae}. It would then be natural
to call the sequence $\binom{n-1}kB_{n-1-k}$ the \emph{Arima
numbers}; see Section~\ref{S:arima} for their distributional aspect.
Informative materials on the life and mathematical works of Matsunaga
can be found in the two books (in Japanese) by Fujiwara \cite
{Fujiwara1957} and by Hirayama \cite{Hirayama1987}, respectively.

Briefly, Yoshisuke Matsunaga (born in 1694? and died in 1744) was a
mathematician in the Edo period ({\footnotesize{江戸時代}}). His
original surname was Terauchi {(\footnotesize{寺内}}), and also known
under a few different names such as Heihachiro
{(\footnotesize{平八郎}}), Gonpei {(\footnotesize{權平}}), and
Yasuemon {(\footnotesize{安右衛門}}); other names used include
Higashioka {(\footnotesize{東岡}}), Tangenshi
{(\footnotesize{探玄子}}), etc. Matsunaga served first in the Arima
family in Kurume Domain ({\footnotesize{久留米藩}}); he also learned
\emph{Wasan} ({\footnotesize{和算}}, Japanese Mathematics) from
Murahide Araki ({\footnotesize{荒木村英}}) who was a disciple of
Takakazu Seki ({\footnotesize{関孝和}}), the founder of modern Wasan.
He then came to Iwakidaira Domain ({\footnotesize{磐城平藩}}) and was
employed by Masaki Naito ({\footnotesize{内藤政樹}}) in 1732. There,
he worked with Yoshihiro Kurushima ({\footnotesize{久留島義太}}), and
his research was believed to be influenced by the theory of Kenko
Takebe ({\footnotesize{建部賢弘}}) and other Seki disciples. He
developed and largely improved Seki's Mathematics. One of his
representative achievements is the calculation of $\pi$  (circumference-diameter ratio of a circle) to 51 digits
(of which the first 49 are correct; see \cite[p.~457]{Fujiwara1957}).
He is also known to compute the series expansions of trigonometric
functions such as sine, cosine, arc-sine, etc. For more information,
see \cite{Fujiwara1957,Hirayama1987}.
See also \href{https://www.ndl.go.jp/math/e/}{this webpage} for more information on \emph{Japanese Mathematics in the Edo Period}.

\subsection{The Genjik\={o} game.}

The set-partition combinatorics developed by the Wasanists in the Edo
period was largely motivated by the parlor game called
\emph{Genjik\={o}} ({\footnotesize{源氏香}}, literally Genji incense,
where Genji refers to the famous novel \emph{Genji Monogatari}
({\footnotesize{源氏物語}}), or \emph{The Tale of Genji} by
Murasaki-Shikibu, {\footnotesize{紫式部}}), as already mentioned in
the combinatorial literature; see \cite{Gould1979, Knuth2011,
Mansour2016} and the webpage \cite{Luschny2011}.

\begin{figure}[!htb]
    \begin{center}
        \begin{tabular}{cc}
        \includegraphics[height=12cm]{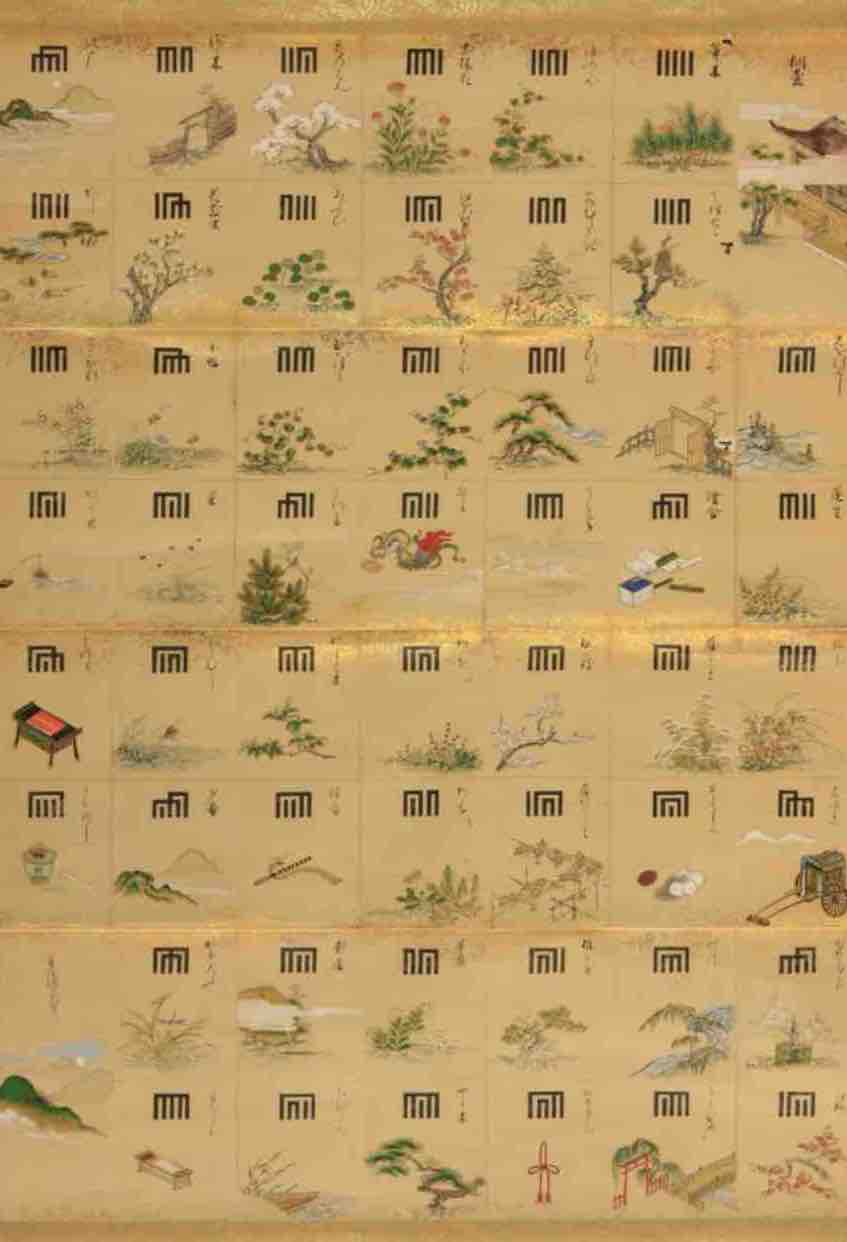}
        \end{tabular}
    \end{center}
    \caption{A Genjik\={o}nozu-painting by Mitsuoki Tosa
	({\footnotesize{土佐光起}}, 1617--1691) collected at Wadeda
	University Library (image courtesy of Waseda University
	Library).}\label{F:oldest-Genjimon}
\end{figure}

According to the preface of Arima's book \cite{Arima1763} (see also
\cite[p.~332]{Hayashi1931}):
\begin{quote}\small
	While Takakazu Seki initiated the study of Genjik\={o}
	combinatorics (or the \emph{techniques of separate-and-link}), it
	was Matsunaga and his contemporary Kurushima who probed its
	origin and developed fundamentally the techniques.
\end{quote}

The game is part of the Japanese K\={o}d\={o} ({\footnotesize{香道}}, 
or the Way of Fragrance, the same character ``k\={o}"
{\footnotesize{香}}, also means fragrance) to appreciate the
fragrances of the incense. It was established in the late Muromachi
period ({\footnotesize{室町時代}}) in the sixteenth century and
became popular in the upper class during the Edo period. 
It consists of the following steps:

\begin{enumerate}
	
\item Five different types of incense sticks are cut into five pieces
each;

\item Five of these 25 pieces are chosen to be smoldered;

\item Guests smell each incense and try to discern (or ``listen to" in
a silent ambience and calm mood) which among the five incenses chosen
are the same and which are different;

\item On the answer sheets, guests write their names and the
conjectured composition of the incenses already smoldered as
\emph{Genjik\={o}nozu} ({\footnotesize{源氏香の図}} or
\emph{Genjimon}, {\footnotesize{源氏紋}}, meaning the \emph{patterns
of Genjik\={o}}). The Genjik\={o}nozu is composed of five vertical
bars for the five incenses in right-to-left order; then link the
vertical bars with a horizontal line on top if the corresponding
incenses are thought to be the same.). See for
example Figure~\ref{F:oldest-Genjimon} for one of the earliest
paintings about Genjik\={o}nozu found so far.

\item The game has no winners or losers; if the answer is correct,
the five-stroke Kanji character {\footnotesize{玉}} (meaning
literally ``jade'' and figuratively something precious or beautiful)
is written on the answer sheet.

\end{enumerate}

In addition to the Genjik\={o}nozus used to represent the patterns of
incenses, a chapter name from \emph{The Tale of Genji} is also
associated with each of the 52 configurations of the five incenses,
not only for easier reference and use but sometimes also for a
reading of that chapter; see Figure~\ref{F:Genjimon1}. Partly
due to such an unusual association and their unusual aesthetic,
cryptic features, the Genjik\={o}nozus continue to be used in the
design of diverse modern applications such as patterns on kimono,
wrapping papers, folding screens, badges, etc.

{
\begin{figure}[!htb]
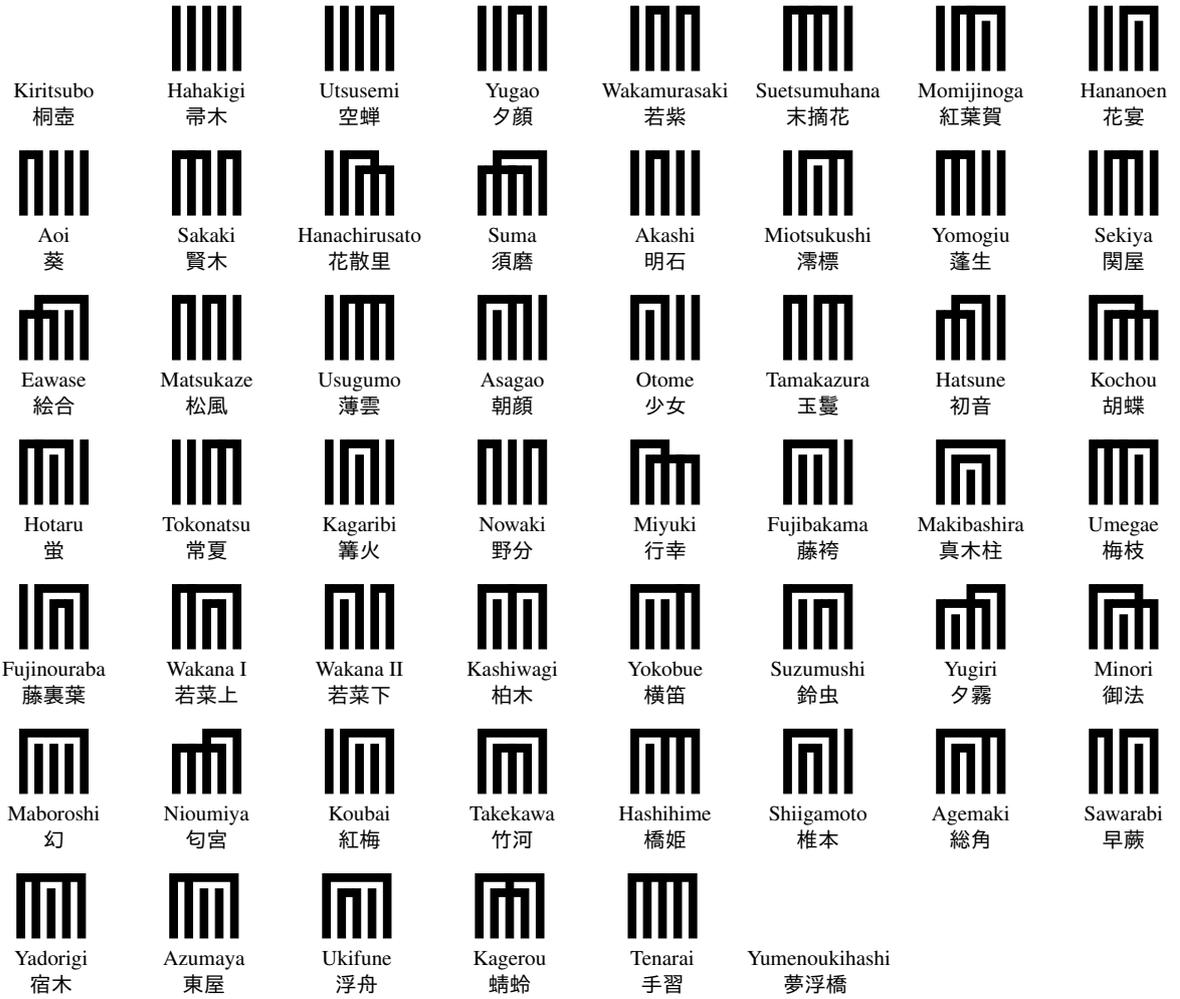
\ifsize
    \centering
\test{{}}{{}}{{Kiritsubo}}{桐壺}
\test{{1, 1, 1, 1, 1}}{{}}{{Hahakigi}}{帚木}
\test{{1, 1, 1, 1, 1}}{{4/5}}{{Utsusemi}}{空蝉}
\test{{1, 1, 1, 1, 1}}{{3/4}}{{Yugao}}{夕顔}
\test{{1, 1, 1, 1, 1}}{{2/3, 4/5}}{{Wakamurasaki}}{若紫}
\test{{1, 1, 1, 1, 1}}{{1/4}}{{Suetsumuhana}}{末摘花}
\test{{1, 1, 1, 2, 1}}{{2/5}}{{Momijinoga}}{紅葉賀}
\test{{1, 1, 1, 2, 1}}{{3/5}}{{Hananoen}}{花宴}
\test{{1, 1, 1, 1, 1}}{{1/2}}{{Aoi}}{葵}
\test{{1, 1, 1, 1, 1}}{{1/3, 4/5}}{{Sakaki}}{賢木}
\test{{1, 1, 2, 1, 2}}{{2/4, 3/5}}{{Hanachirusato}}{花散里}
\test{{2, 1, 2, 2, 1}}{{2/5, 1/4}}{{Suma}}{須磨}
\test{{1, 1, 1, 1, 1}}{{2/3}}{{Akashi}}{明石}
\test{{1, 1, 2, 1, 1}}{{2/5}}{{Miotsukushi}}{澪標}
\test{{1, 1, 1, 1, 1}}{{1/3}}{{Yomogiu}}{蓬生}
\test{{1, 1, 1, 1, 1}}{{2/4}}{{Sekiya}}{関屋}
\test{{2, 1, 2, 2, 1}}{{1/3, 2/5}}{{Eawase}}{絵合}
\test{{1, 1, 1, 1, 1}}{{1/2, 3/4}}{{Matsukaze}}{松風}
\test{{1, 1, 1, 1, 1}}{{2/5}}{{Usugumo}}{薄雲}
\test{{1, 2, 1, 1, 1}}{{1/4}}{{Asagao}}{朝顔}
\test{{1, 2, 1, 1, 1}}{{1/3}}{{Otome}}{少女}
\test{{1, 1, 1, 1, 1}}{{1/2, 3/5}}{{Tamakazura}}{玉鬘}
\test{{2, 1, 2, 1, 1}}{{1/3, 2/4}}{{Hatsune}}{初音}
\test{{1, 2, 2, 1, 2}}{{1/4, 2/5}}{{Kochou}}{胡蝶}
\test{{1, 1, 2, 1, 1}}{{1/4}}{{Hotaru}}{蛍}
\test{{1, 1, 1, 1, 1}}{{3/5}}{{Tokonatsu}}{常夏}
\test{{1, 1, 2, 1, 1}}{{2/4}}{{Kagaribi}}{篝火}
\test{{1, 1, 1, 1, 1}}{{1/2, 4/5}}{{Nowaki}}{野分}
\test{{1, 2, 1, 2, 2}}{{1/3, 2/5}}{{Miyuki}}{行幸}
\test{{1, 2, 2, 1, 1}}{{1/4}}{{Fujibakama}}{藤袴}
\test{{1, 2, 3, 2, 1}}{{1/5, 2/4}}{{Makibashira}}{真木柱}
\test{{1, 1, 1, 2, 1}}{{1/5}}{{Umegae}}{梅枝}
\test{{1, 1, 2, 2, 1}}{{2/5, 3/4}}{{Fujinouraba}}{藤裏葉}
\test{{1, 1, 2, 2, 1}}{{1/5, 3/4}}{{Wakana I}}{若菜上}
\test{{1, 2, 1, 1, 1}}{{1/3, 4/5}}{{Wakana II}}{若菜下}
\test{{1, 2, 1, 2, 1}}{{1/5}}{{Kashiwagi}}{柏木}
\test{{1, 2, 2, 1, 1}}{{1/5}}{{Yokobue}}{横笛}
\test{{1, 2, 2, 2, 1}}{{1/5, 3/4}}{{Suzumushi}}{鈴虫}
\test{{2, 3, 1, 2, 1}}{{1/4, 3/5}}{{Yugiri}}{夕霧}
\test{{1, 2, 3, 1, 2}}{{1/4, 2/5}}{{Minori}}{御法}
\test{{1, 2, 2, 2, 1}}{{1/5}}{{Maboroshi}}{幻}
\test{{2, 2, 1, 2, 1}}{{1/4, 3/5}}{{Nioumiya}}{匂宮}
\test{{1, 1, 2, 2, 1}}{{2/5}}{{Koubai}}{紅梅}
\test{{1, 2, 2, 2, 1}}{{1/5, 2/4}}{{Takekawa}}{竹河}
\test{{1, 2, 1, 1, 1}}{{1/5}}{{Hashihime}}{橋姫}
\test{{1, 2, 2, 1, 1}}{{1/4, 2/3}}{{Shiigamoto}}{椎本}
\test{{1, 2, 2, 1, 1}}{{1/5, 2/3}}{{Agemaki}}{総角}
\test{{1, 1, 1, 2, 1}}{{1/2, 3/5}}{{Sawarabi}}{早蕨}
\test{{1, 1, 2, 1, 1}}{{1/5}}{{Yadorigi}}{宿木}
\test{{1, 1, 2, 2, 1}}{{1/5}}{{Azumaya}}{東屋}
\test{{1, 2, 2, 2, 1}}{{1/5, 2/3}}{{Ukifune}}{浮舟}
\test{{1, 2, 1, 2, 1}}{{1/5, 2/4}}{{Kagerou}}{蜻蛉}
\test{{1, 1, 1, 1, 1}}{{1/5}}{{Tenarai}}{手習}
\test{{}}{{}}{$\text{Yumenoukihashi}$}{夢浮橋}
\test{{}}{{}}{$\ $}{}
\test{{}}{{}}{$\ $}{}
\caption{The 52 Genjik\={o} patterns arranged in the chapter order
(top-down and left-to-right) of the novel Genji Monogatari (54
chapters).} \label{F:Genjimon1}
\end{figure}
}

\section{Matsunaga's procedure to compute $B_n$}
\subsection{Matsunaga's 1726 book \cite{Matsunaga1726}}

Matsunaga wrote voluminously on a broad range of topics in Wasan
(more than 50 titles listed in Hirayama's book \cite{Hirayama1987}),
but none of these was printed and published during his life time due
partly to the tradition at that time; it is therefore common to find
different hand-written versions of the same book.

Among his books the one dated 1726 and entitled \emph{Danren
S\={o}jutsu} ({\footnotesize{断連総術}}, literally \emph{General
techniques of separate-and-link}) \cite{Matsunaga1726} (with a total
of 11 double-pages) is indeed completely devoted to the calculation
of Bell numbers, aiming specially to enumerate all possible ways to
connect and to separate a given number of incenses in the Genjik\={o}
game.

Our first aim in this paper is to provide more details contained in
this book \cite{Matsunaga1726}, and to highlight his procedure to
compute Bell numbers, which is nevertheless not the equation
\eqref{E:bell-formulae} as most authors believed and referenced. This
book, as well as a few others mentioned in this paper, is freely
available at the webpage of
\href{https://www.nijl.ac.jp/en/}{National Institute of Japanese
Literature} (NIJL). Since all these ancient books do not have page
numbers (and the page numbers differ in different versions of the
same book), all page numbers referenced in this paper are indeed the
(double) page order of the corresponding digital file at the NIJL
database. We will provide the corresponding URLs whenever possible.

Matsunaga's procedure was later examined and improved in great detail in
Arima's 1763 book \cite{Arima1763} on ``\emph{Variational techniques
of separate-and-link}'' ({\footnotesize{断連変局法}}, about 55
double-pages), and there we find the occurrence of the recurrence
\eqref{E:bell-formulae}; see Section~\ref{S:arima} for more details.

\subsection{Matsunaga numbers}
Let $s_{n,k}$ denote (signed) Stirling numbers of the first kind (see
\href{https://oeis.org/A008275}{A008275} and
\href{https://oeis.org/A048994}{A048994}), and $\stirling{n}{k}
= (-1)^{n-k}s_{n,k}$ the unsigned version (see
\href{https://oeis.org/A132393}{A132393}).
Denote by $[z^n]f(z)$ the coefficient of $z^n$ in the Taylor expansion
of $f(z)$.
\begin{thm}[Matsunaga, 1726 \cite{Matsunaga1726}]\label{T:Matsunaga}
For $n\ge1$
\begin{align}\label{E:matsunaga}
    B_n = 1+\frac1{n!}\sum_{1\le k\le n} M_{n,k} n^k,
\end{align}
where the Matsunaga numbers $M_{n,k}$ are defined recursively by
\begin{align}\label{E:bnk}
    M_{n,k} = nM_{n-1,k} + \beta_ns_{n,k}
\end{align}
for $1\le k\le n$ with the boundary conditions $M_{n,k}=0$ for
$n\le1$, $k\le0$ and $k>n$. Here $\beta_n := n![z^n]e^{e^z-1-z}$
counts the number of set partitions of $n$ elements without
singletons.
\end{thm}

\begin{proof}
A direct iteration of $\eqref{E:bnk}$ gives
\begin{align}\label{E:ank-stir}
    M_{n,k} = n!\sum_{k\le j\le n}\frac{\beta_j}{j!}\, s_{j,k}.
\end{align}
Then, by the defining relation of Stirling numbers of the first kind,
\begin{align*}
    \sum_{1\le k\le m} s_{m,k}z^k
    = z(z-1)\cdots(z-m+1),
\end{align*}
we have
\begin{equation*}
\begin{split}
    1+\frac1{n!}\sum_{1\le k\le n}M_{n,k}n^k
    &= 1+\sum_{1\le k\le n}n^k
    \sum_{k\le j\le n}\frac{\beta_j}{j!}\, s_{j,k} \\
    &= 1+\sum_{1\le j\le n}\frac{\beta_j}{j!}
    \sum_{1\le k\le j}s_{j,k}n^k \\
    &= 1+\sum_{1\le j\le n}\frac{\beta_j}{j!}\,
    n(n-1)\cdots(n-j+1)\\
    &= n!\sum_{0\le j\le n}
    \frac{\beta_j}{j!}\cdot \frac1{(n-j)!}\\
    &= n![z^n]e^{e^z-1-z}\cdot e^z = B_n,
\end{split}
\end{equation*}
with $\beta_0=1$. Combinatorially, the last relation is equivalent to
splitting set partitions into those with blocks of size $1$ and those
without.
\end{proof}

Note that $B_n = \beta_{n+1}+\beta_n$, which, by a direct iteration,
gives
\begin{align}\label{E:beta-bell}
	\beta_n = \sum_{0\le j<n}(-1)^{n-1-j}B_j+(-1)^n
	\qquad(n\ge0).
\end{align}
The first few terms of $\beta_n$ are given by (see
\href{https://oeis.org/A000296}{OEIS A000296})
\[
    \{\beta_n\}_{n\ge1}
    = \{0,1,1,4,11,41,162,715, 3425, 17722, 98253, 580317,\dots\},
\]
and Table~\ref{T:ank} gives the first few rows of $M_{n,k}$ for
$n=1,\dots,7$; these numbers already appeared in \cite{Matsunaga1726}
and \cite{Arima1763}.
\begin{table}[!ht]
\begin{center}
\begin{tabular}{c|rrrrrrr}
	$n\backslash k$
	& $1$ & $2$ & $3$ & $4$ & $5$ & $6$ & $7$ \\ \hline
    $1$ & $0$ &&&&& \\
    $2$ & $-1$ & $1$ &&&& \\
    $3$ & $-1$ & $0$ & $1$ &&& \\
	$4$ & $-28$ & $44$ & $-20$ & $4$ && \\
    $5$ & $124$ & $-330$ & $285$ & $-90$ & $11$ & \\
    $6$ & $-4176$ & $9254$ & $-7515$ & $2945$ & $-549$ & $41$ \\
    $7$ & $87408$ & $-220990$ & $210483$ & $-98455$ & $24507$
    & $-3115$ & $162$
\end{tabular}
\end{center}
\caption{The values of $M_{n,k}$ for $n=1,\dots,7$ and $1\le k\le n$.
Each row sum is zero.}\label{T:ank}
\end{table}

\subsection{The ``non-conventional'' procedure to compute $B_n$}
Matsunaga's extraordinary procedure to compute $B_n$ is then as follows
(extraordinary in the sense that we have not found it in the combinatorial
literature).
\begin{itemize}
    \item Tabulate first $s_{n,k}$ \cite[pp.~\href{https://kotenseki.nijl.ac.jp/biblio/100235422/viewer/2}{2--3}]{Matsunaga1726};

    \item Compute $\beta_n$ by a direct exhaustive combinatorial
    enumeration \cite[pp.~\href{https://kotenseki.nijl.ac.jp/biblio/100235422/viewer/3}{3--6}]{Matsunaga1726} and
    \cite[pp.~\href{https://kotenseki.nijl.ac.jp/biblio/100235422/viewer/8}{8--10}]{Matsunaga1726};

    \item Tabulate $M_{n,k}$ by using the recurrence \eqref{E:bnk}
    \cite[p.~\href{https://kotenseki.nijl.ac.jp/biblio/100235422/viewer/7}{7}]{Matsunaga1726};

    \item Evaluate the polynomial (in $n$) $\sum_k M_{n,k}n^k$ by
    Horner's rule \cite[p.~\href{https://kotenseki.nijl.ac.jp/biblio/100235422/viewer/7}{7}]{Matsunaga1726}:
    \[
        \sum_{1\le k\le n} M_{n,k}n^k
        = n\lpa{M_{n,1}+n\lpa{M_{n,2}+\cdots
        +n(M_{n,n-1}+M_{n,n}n)}};
    \]
    \item Then dividing the above sum by $n!$ and adding $1$ yields
    $B_n$, where the values $B_2,\dots, B_8$ are given
    \cite[pp.~\href{https://kotenseki.nijl.ac.jp/biblio/100235422/viewer/10}{10--11}]{Matsunaga1726}.
\end{itemize}
Note that the origin of Horner's rule can be traced back to Chinese
and Persian Mathematics in the thirteenth century.

In particular,
\begin{small}
\begin{equation}\label{E:Bn-small}
\begin{split}
    B_2 &= 1+\frac1{2!}\cdot 2(\underline{-1}+\underline{1}\cdot 2)
    = 1+\frac{2\cdot 1}{2}=2\\
    B_3 &= 1+\frac1{3!}\cdot 3(\underline{-1}
    +3(\underline{0}+\underline{1}\cdot 3))
    = 1+\frac{3\cdot 8}{6} = 5\\
    B_4 &= 1+\frac1{4!}\cdot 4(\underline{-28}
    +4(\underline{44}+4(\underline{-20}+\underline{4}\cdot 4)))
    = 1+\frac{4\cdot 84}{24} = 15\\
    B_5 &= 1+\frac1{5!}\cdot 5(\underline{124}
    +5(\underline{-330}+5(\underline{285}
    +5(\underline{-90}+\underline{11}\cdot 5))))
    = 1+\frac{5\cdot 1224}{120} = 52,
\end{split}
\end{equation}
\end{small}
where the Matsunaga numbers (see Table~\ref{T:ank}) are underlined.

\begin{table}[!ht]
\begin{center}
\begin{tabular}{c|rrrrrr|c}
	$n\backslash k$
	& $1$ & $2$ & $3$ & $4$ & $5$ & $6$ & $B_n$ \\ \hline
    $2$ & $-2$ & $4$ &&&&& $2$ \\
    $3$ & $-3$ & $0$ & $27$ &&&& $5$\\
	$4$ & $-112$ & $704$ & $-1280$ & $1024$ &&& $15$\\
    $5$ & $620$ & $-8250$ & $35625$ & $-56250$ & $34375$ && $52$ \\
    $6$ & $-25056$ & $333144$ & $-1623240$ & $3816720$
    & $-4269024$ & $1912896$ & $203$
\end{tabular}
\end{center}
\caption{The values of $M_{n,k}n^k$ for $n=2,\dots,6$ and
$1\le k\le n$; they are mostly much larger in modulus than
$B_n$}\label{T:ank-nk}
\end{table}

\subsection{First appearance of $s_{n,k}$}
On the other hand, a careful reader will notice the time difference
between Matsunaga's use of $s_{n,k}$ in his 1726 book and Stirling's
introduction of these numbers in 1730. Indeed, Stirling numbers of
the first kind already appeared earlier in Seki's posthumous book
``\emph{Katsuy\=o Sanp\=o}'' ({\footnotesize{括要算法}},
\emph{Compendium of Arithmetics} \cite{Seki1712}) published in 1712, and thus were not new to the followers
of the Seki School. In particular, the recurrence formula
\begin{align*}
    s_{n,k} = s_{n-1,k-1}-(n-1)s_{n-1,k}
\end{align*}
is given in Seki's book (see Figure~\ref{F:seki-p33}). On the other hand, it is also known that
Stirling numbers of the first kind can be traced earlier back to
Harriot's manuscripts near the beginning of the 17th century; see
\cite{Knuth1992}.

\begin{figure}[!ht]
\begin{center}
\begin{minipage}[t]{0.45\textwidth}
\vspace{0pt}\hspace{0.5cm}
\includegraphics[height=5.5cm]{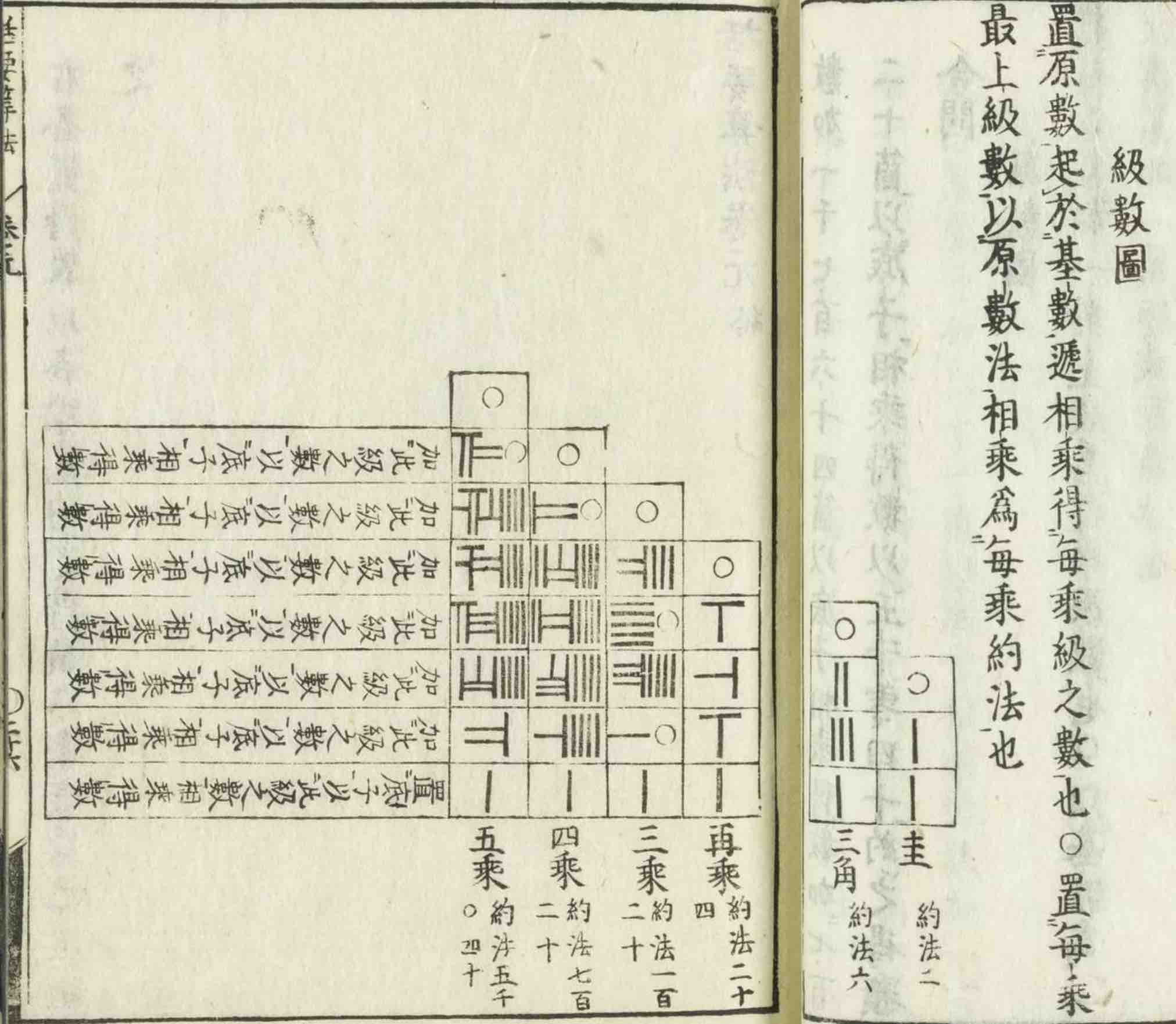}
\end{minipage}\;
\begin{minipage}[t]{0.45\textwidth}
\vspace{0.7cm}
\begin{tabular}{rrrrrr}
$0$ \\
$720$ & $0$ \\
$1764$ & $120$ & $0$ \\
$1624$ & $274$ & $24$ & $0$ \\
$735$ & $225$ & $50$ & $6$ & $0$ \\
$175$ & $85$ & $35$ & $11$ & $2$ & $0$ \\
$21$ & $15$ & $10$ & $6$ & $3$ & $1$ \\
$1$ & $1$ & $1$ & $1$ & $1$ & $1$ \\ \hline
$5040$ & $720$ & $120$& $\ 24$ & $\ \ 6$ & $\ \ 2$
\end{tabular}
\end{minipage}
\end{center}
\caption{Left: page \href{https://dl.ndl.go.jp/info:ndljp/pid/3508173/33}{33} (image from the digital collections of The National Diet Library of Japan) of Seki's book \cite{Seki1712} where
Stirling numbers of the first kind appear; the recurrence is
described on the leftmost margin. Right: the unsigned
Stirling numbers as displayed in Seki's table written in their modern
forms.}\label{F:seki-p33}
\end{figure}

Due to their diverse use in different contexts, the
Stirling numbers $s_{n,k}$ (\href{https://oeis.org/A008275}{A008275})
have a large number of variants; common ones collected in the OEIS
include the following ones:
\begin{itemize}
    \item additional zeroth column:
    (\href{https://oeis.org/A048994}{A048994}),
    \item unsigned version: $|s_{n,k}|$
    (\href{https://oeis.org/A130534}{A130534}),
    \item row-reversed: $s_{n,n+1-k}$
    (\href{https://oeis.org/A008276}{A008276}),
    \item multiples: $(n+1)\stirling{n}{k}$
    (\href{https://oeis.org/A180013}{A180013}), which enumerates the
    number of fixed points in permutations of $n$ elements with $k$
    cycles,
    \item unsigned, row-reversed: $|s_{n,n+1-k}|$
    (\href{https://oeis.org/A094638}{A094638}),
    \item unsigned, additional zeroth column:
    (\href{https://oeis.org/A132393}{A132393}).
\end{itemize}

\subsection{Inefficiency of the procedure}
While the proof of \eqref{E:matsunaga} is straightforward, the fact
that the circuitous expression \eqref{E:matsunaga} represents by far
the earliest way to compute $B_n$ comes as a surprise. Even more
surprising is that such a procedure is indeed extremely
\emph{inefficient} from a computational viewpoint, since it involves
computing large numbers with alternating signs (see
Table~\ref{T:ank-nk} and \eqref{E:Bn-small}), resulting in violent
cancellations as $n$ increases. In fact, since Stirling numbers of
the first kind may grow as large as $n!/\sqrt{\log n}$ in absolute
value near their peaks (see \cite{Hwang1995} or \eqref{E:mw1} below)
and $n^k$ is also factorially large for linear $k$, the cancellations
occur indeed at a \emph{factorial scale}. This might also explain why
Matsunaga computed $B_n$ only for $n\le 8$ in the end of
\cite{Matsunaga1726}. On the other hand, from \eqref{E:matsunaga}, we
see that
\[
    \sum_{1\le k\le n}M_{n,k}n^k = (B_n-1)n!.
\]
Thus if no better numerical procedure is introduced to handle the
calculation of the normalized sum $\sum_{1\le k\le n}M_{n,k}n^k/n!$,
then intermediate steps will involve numbers as large as $B_n\times
n!$, which grows like
\[
    \log \Lpa{\sum_{1\le k\le n}M_{n,k}n^k}
    \sim \log B_n n! \sim 2n\log n - n\log \log n -n,
\]
(see \eqref{E:Bn1} below), which is close to $2\log n!$.

So one naturally wonders whether Matsunaga was led to this procedure
by roughly following backward the same reasoning as that given in the
proof of
\eqref{E:matsunaga}. The inefficiency was noticed in later works such
as in Arima's book \cite{Arima1763}, and there the more efficient
recurrence \eqref{E:bell-formulae} is given. 
\section{Distribution of Matsunaga numbers}

The above numerical viewpoint for Matsunaga's procedure naturally
motivates the question: how does the numbers $|M_{n,k}|$ distribute
for large $n$ and varying $k$? This question also has its own
interest \emph{per se} in view of the historical importance of this
sequence.

\subsection{Identities}

We derive first a more practical expression for $|M_{n,k}|$ because
the sum expression \eqref{E:ank-stir} becomes less convenient with
the absolute value sign.

\begin{lmm} For $1\le k\le n$, $n\ge1$ and $(n,k)\ne(3,1)$,
\begin{align}\label{E:abs-Mnk}
    |M_{n,k}|
    = n!\sum_{k\le j \le n}\frac{(-1)^{n-j}\beta_j}{j!}
    \,|s_{j,k}|.
\end{align}	
\end{lmm}
\begin{proof}  	
First, by the alternating-sign relation $|s_{j,k}| = (-1)^{j+k}s_{j,k}$,
\begin{align}\label{E:Mnk-w-abs}
    \frac{|M_{n,k}|}{n!}
	= \left|\sum_{k\le j \le n}\frac{\beta_j}{j!}
    \,s_{j,k} \right|
    = \pm\sum_{k\le j \le n}\frac{(-1)^{n-j}\beta_j}{j!}
    \,|s_{j,k}|.
\end{align}
For $0\le n\le 3$ and $1\le k\le n$, it is easily checked that the
equality holds with a plus sign in front of the second sum in
\eqref{E:Mnk-w-abs} except for the pair $(n,k)=(3,1)$ where the two
sides of \eqref{E:Mnk-w-abs} differ by a minus sign. We now show that
the sequence $\{\beta_n |s_{n,k}|/n!\}_{n}$ is nondecreasing for $n\ge \max\{k,4\}$ and fixed $k\ge1$, implying that
\eqref{E:Mnk-w-abs} holds with the plus sign and in turn
that \eqref{E:abs-Mnk} is valid. Since
\begin{align}\label{E:abs-stir1-rr}
	|s_{n,k}| = |s_{n-1,k-1}|+(n-1)|s_{n-1,k}|,
\end{align}
we have the trivial inequality $|s_{n+1,k}|\ge n|s_{n,k}|$; thus it
suffices to prove that
\begin{align}\label{E:beta-ineq}
    \frac{\beta_{n+1}}{n+1}\ge \frac{\beta_n}{n},\qquad(n\ge3).
\end{align}
For that purpose, we use the recurrence
\[
    \beta_{n+1} = \sum_{0\le j\le n-1}\binom{n}{j}\beta_j
	\qquad(n\ge1),
\]
with the initial conditions $\beta_0=1$ and $\beta_1=0$, which is
obtained by taking derivative of the exponential generating function (EGF)
$e^{e^z-1-z}$ of $\beta_n$ and by equating the coefficients. Then \eqref{E:beta-ineq} follows
from the inequality
\[
    \binom{n}{j}= \frac{n}{n-j}\binom{n-1}{j}
    \ge \frac{n+1}{n}\binom{n-1}{j}\qquad(1\le j<n).
    \qedhere
\]
\end{proof}

Define $P_n(v) := \sum_{1\le k\le n}|M_{n,k}|v^k$.
\begin{prop} For $n\ge4$,
\begin{align}\label{E:Pnv}
    \frac{P_n(v)}{n!}
	= \sum_{0\le j\le n-2}\binom{v+n-j-1}{n-j}
    (-1)^{j}\beta_{n-j}\qquad(n\ge4).
\end{align}	
\end{prop}
\begin{proof}
This follows from \eqref{E:abs-Mnk} and the generating polynomial
\[
    \sum_{1\le k\le n}|s_{n,k}|z^k
    = z(z+1)\cdots(z+n-1).\qedhere
\]
\end{proof}

\begin{cor}
For $n\ge4$
\[
    \frac{P_n(1)}{n!} = \sum_{0\le j\le n-2}
    (-1)^j\beta_{n-j} = \sum_{0\le j<n}
    (-1)^j(j+1)B_{n-1-j} + (-1)^nn.
\]
\end{cor}
\begin{proof}
Substitute $v=1$ in \eqref{E:Pnv} and then use \eqref{E:beta-bell}.
\end{proof}
By the ordinary generating function of the Bell numbers, we also have
the relation
\[
	\frac{P_n(1)}{n!}
	= [z^n]\frac{z}{(1+z)^2}\sum_{k\ge1}\prod_{1\le j\le k}
	\frac{z}{1-jz}\qquad(n\ge4).
\]
Note that $P_n(1) = \sum_{1\le j\le n}|M_{n,j}|\sim \beta_nn!$ while
$\sum_{1\le j\le n}M_{n,j}=0$.

\subsection{Asymptotics}

In this section, we turn to the asymptotics and show that $|M_{n,k}|$
is very close to $\beta_n |s_{n,k}|$ for large $n$, so the
distributional properties of $|M_{n,k}|$ will mostly follow from
those of $|s_{n,k}|$.

\begin{prop}\label{P:en} Uniformly for $1\le k\le n$
\begin{align}\label{E:bnk-ae}
    M_{n,k} = \beta_n s_{n,k}\lpa{1+O\lpa{n^{-1}\log n}}.
\end{align}
\end{prop}
Thus only the term with $j=n$ in the sum expression
\eqref{E:ank-stir} is dominant for large $n$; see
Figure~\ref{F:bnk-stir-ratio}.
\begin{figure}[!ht]
    \begin{center}
        \includegraphics[height=4cm]{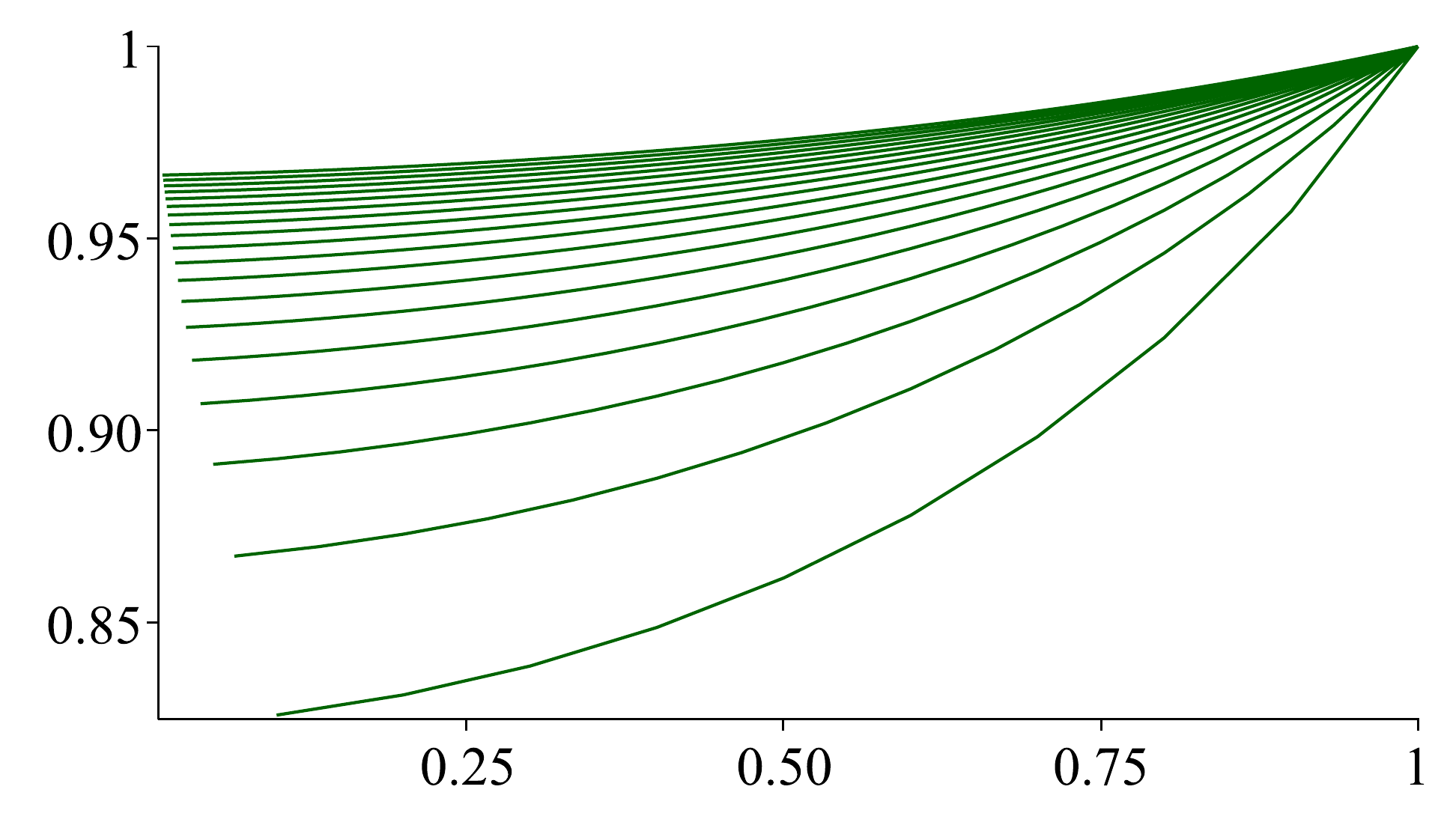}
    \end{center}
    \caption{Uniformity of the ratios of $M_{n,k}$ to
    $\beta_ns_{n,k}$ for $1\le k\le n$ and $n=10,15,\dots,100$ (in
	bottom-up order). The approximation is seen to be good even for moderate values of $n$.}
    \label{F:bnk-stir-ratio}
\end{figure}

\subsubsection{Asymptotics of $\beta_n$}

Recall that $\beta_n := n![z^n]e^{e^z-1-z}$ denotes the number of set
partitions without singletons. Let $W(x)>0$ denote the (principal
branch of) Lambert $W$-function, which satisfies the equation $We^W=x$
(and is positive for positive $x$). Asymptotically, for large $x$
\[
    W(x) = \log x - \log\log x + \frac{\log \log x}{\log x}
    + O\lpa{(\log x)^{-2}(\log\log x)^2};
\]
see Corless et al.'s survey paper ~\cite{Corless1996} for more
information on $W(x)$.

\begin{lmm} For large $n$ ($w := W(n)$)
\begin{align}\label{E:en}
    \beta_n = \frac{e^{(w+w^{-1}-1)n-w-1}}{\sqrt{w+1}}
    \Lpa{1-\frac{26w^4+67w^3+46w^2}{24n(w+1)^3}
	+O\lpa{n^{-2}(\log n)^2}}.
\end{align}
\end{lmm}
\begin{proof}
This follows from applying the standard saddle-point method to
Cauchy's integral formula:
\[
    \beta_n = \frac{n!}{2\pi i}\oint_{|z|=r_0}
    z^{-n-1}e^{e^z-1-z}\dd z,
\]
where $r_0=r_0(n)>0$ solves the saddle-point equation
\[
    r_0(e^{r_0}-1)=n,
\]
which satisfies asymptotically (by Lagrange inversion formula
\cite[A.6]{Flajolet2009})
\[
    r_0= w+ \sum_{j\ge1}\eta_j n^{-j},
    \quad \text{with}\quad
    \eta_j = \frac1j[t^{j-1}]\Lpa{\frac{wt(w+t)}
    {(w+t)e^t-w}}^j.
\]
In particular,
\[
    r_0 = w + \frac{w^2}{n(w+1)}
    -\frac{w^3(w^2-2)}{2n^2(w+1)^3}
    +O\lpa{n^{-3}w^3}.
\]
See \cite{deBruijn1981}, \cite{Moser1955}, or \cite{Flajolet2009} for
similar details concerning Bell numbers or the saddle-point method.
\end{proof}
\begin{cor} For large $n$ and $\ell=O(1)$
\begin{align}\label{E:en-ratio}
    \frac{\beta_{n-\ell}}{\beta_n}
    = \Lpa{\frac{W(n)}n}^\ell\lpa{1+O\lpa{n^{-1}\ell^2\log n}}.
\end{align}
\end{cor}
\begin{proof}
By the saddle-point approximation \eqref{E:en} and the expansion
\begin{align}\label{E:Wn1}
    W(n-t) = W(n) - \frac{W(n)t}{n(W(n)+1)}
    +O\lpa{n^{-2}t^2},
\end{align}
for $t=O(1)$, where the relation $W'(x) = \frac{W(x)}{x(W(x)+1)}$ is
used; see \cite[Eqs.~(3.2) and (3.3)]{Corless1996}.
\end{proof}

\subsubsection{Proof of Proposition~\ref{P:en}}
We are now ready to prove Proposition~\ref{P:en}.
\begin{proof}
From \eqref{E:ank-stir}, we have
\[
    \frac{M_{n,k}}{\beta_n}
    = s_{n,k}\llpa{1+O\llpa{
    \sum_{1\le j\le n-k}\frac{n!}{(n-j)!}
    \cdot \frac{\beta_{n-j}}{\beta_n}
    \cdot \frac{|s_{n-j,k}|}{|s_{n,k}|}}}.
\]
Now by the recurrence \eqref{E:abs-stir1-rr}, we have the trivial
inequalities
\[
    |s_{n,k}| \ge (n-1)|s_{n-1,k}|\ge (n-1)(n-2)|s_{n-2,k}|
    \ge\cdots\ge (n-1)\cdots(n-j)|s_{n-j,k}|,
\]
which then give
\begin{align*}
    \sum_{1\le j\le n-k}\frac{n!}{(n-j)!}
    \cdot \frac{\beta_{n-j}}{\beta_n}
    \cdot \frac{|s_{n-j,k}|}{|s_{n,k}|}
    &\le \sum_{1\le j\le n-k}
    \frac{\beta_{n-j}}{\beta_n}
    \cdot \frac{n}{n-j}\\
    &= O\Lpa{\frac{\log n}{n}},
\end{align*}
by the relation \eqref{E:en-ratio}. This completes the proof of
\eqref{E:bnk-ae}.
\end{proof}
Extending the same proof shows that \eqref{E:abs-Mnk} is itself an
asymptotic expansion for large $n$ and each $k\in[1,n-1]$, namely,
\[
    \frac{|M_{n,k}|}{\beta_n |s_{n,k}|}
    = 1+\sum_{1\le j< J}(-1)^j
    \frac{n!}{(n-j)!}\cdot \frac{\beta_{n-j}}{\beta_n}
    \cdot\frac{|s_{n-j,k}|}{|s_{n,k}|}
    +O\lpa{n^{-J}(\log n)^J},
\]
for any bounded $J\ge 1$ satisfying $J\le n-k$.

\subsection{Asymptotic distributions}

With the closed-form \eqref{E:Pnv} and the uniform approximation
\eqref{E:bnk-ae} available, all distributional properties of
$|M_{n,k}|$ can be translated into those of $|s_{n,k}|$. For
simplicity, we introduce the following notation and say that
$\{a_{n,k}\}_k$ \emph{satisfies the local limit theorem (LLT):}
\begin{align*}
    \{a_{n,k}\}_k \simeqq \LLT\lpa{\mu_n, \sigma^2_n; \ve_n},
\end{align*}
\emph{for large $n$ and positive sequence $a_{n,k}$ if the underlying
sequence of random variables}
\[
    \mathbb{P}(X_n=k)
    := \frac{a_{n,k}}{\sum_{j}a_{n,j}},
\]
\emph{satisfies $\mathbb{E}(X_n)\sim \mu_n$, $\mathbb{V}(X_n)\sim
\sigma^2_n$, and}
\[
    \sup_{x\in\mathbb{R}}
    \sigma_n\Bigl|\mathbb{P}(X_n=\mu_n + x\sigma_n)
    -\frac{e^{-\frac12x^2}}{\sqrt{2\pi}}\Bigr| = O(\ve_n),
\]
\emph{with $\sigma_n\to\infty$ and $\ve_n\to0$.} When the convergence
rate is immaterial, we also write $\LLT(\mu_n,\sigma_n^2) =
\LLT(\mu_n,\sigma_n^2;o(1))$. Similarly, the notations
$\mathscr{N}(\mu_n,\sigma_n^2)$ and $\mathscr{N}(\mu_n,\sigma_n^2;
\ve_n)$ denote the \emph{central limit theorem} (CLT or \emph{weak
convergence to standard normal law}) with convergence rates $o(1)$
and $\ve_n$, respectively.

\begin{thm} \label{T:Matsunaga-numbers}
\begin{align}\label{E:Mnk-llt}
    \{|M_{n,k}|\}_k
    \simeqq \LLT\lpa{\log n, \log n; (\log n)^{-\frac12}}.
\end{align}
\end{thm}
See \cite{Hwang1994,Hwang1995,Pavlov1989} for other properties of the
Stirling cycle distribution.

\begin{figure}[!ht]
\begin{center}
\begin{tabular}{cccc}
    \includegraphics[height=3cm]{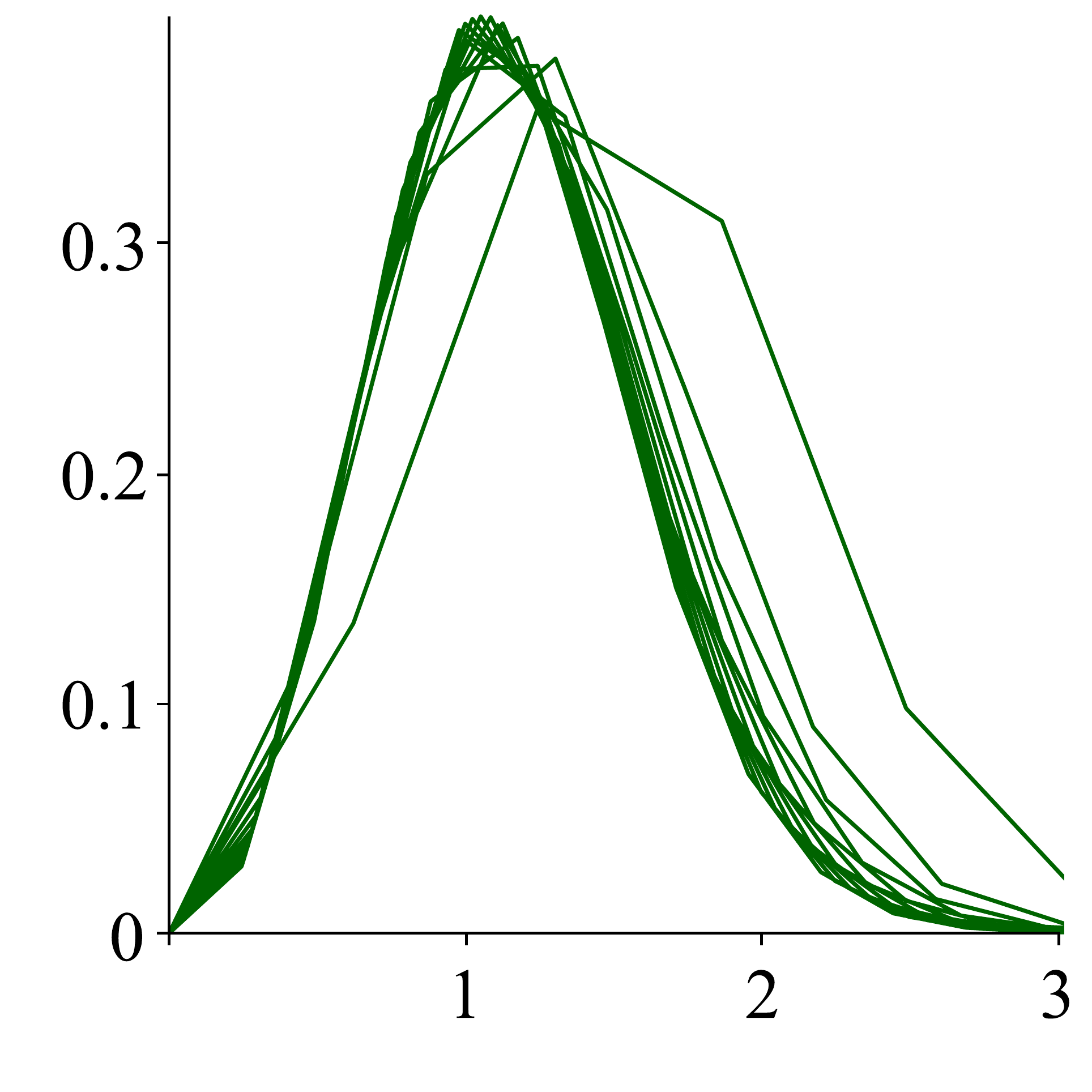} &
    \includegraphics[height=3cm]{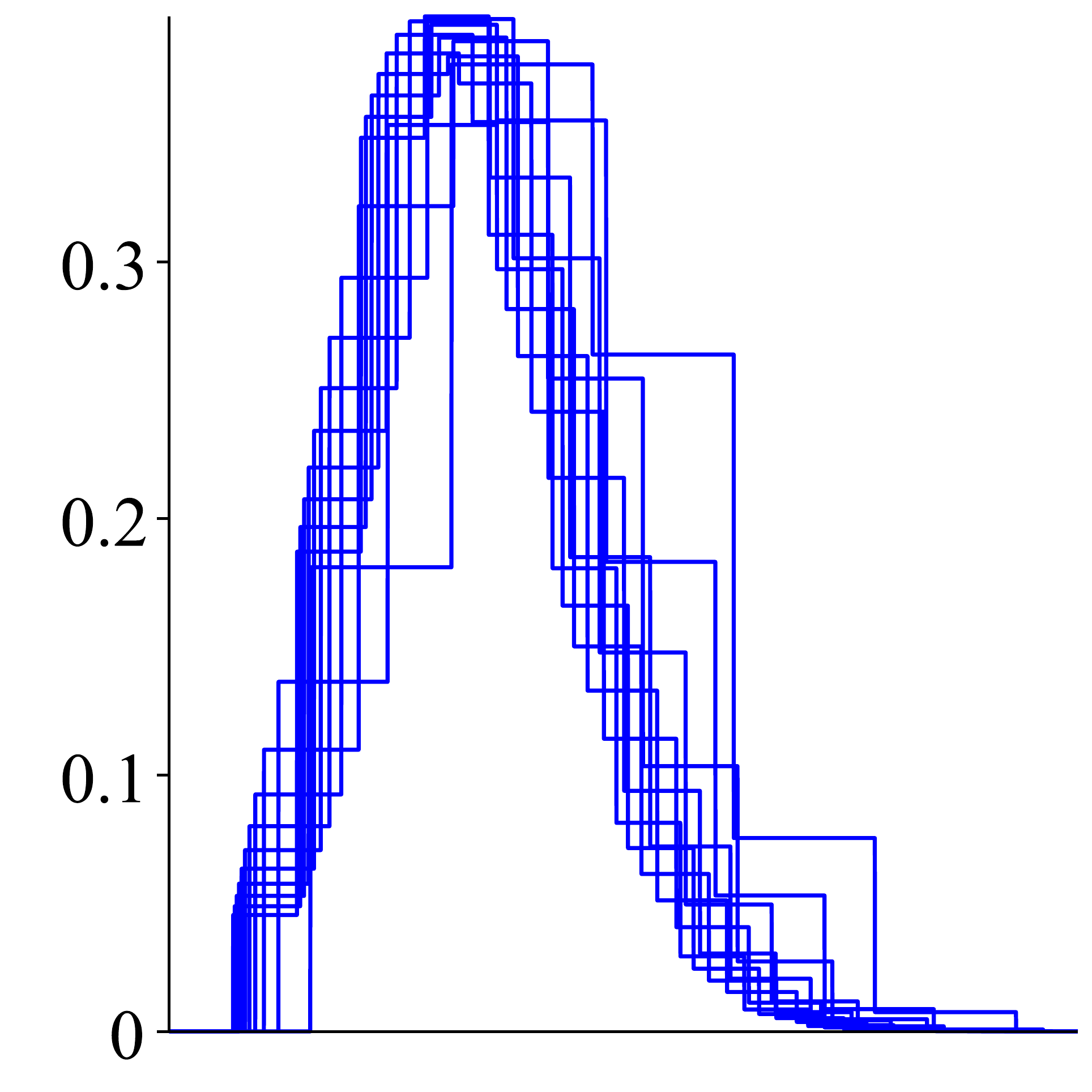} &
    \includegraphics[height=3cm]{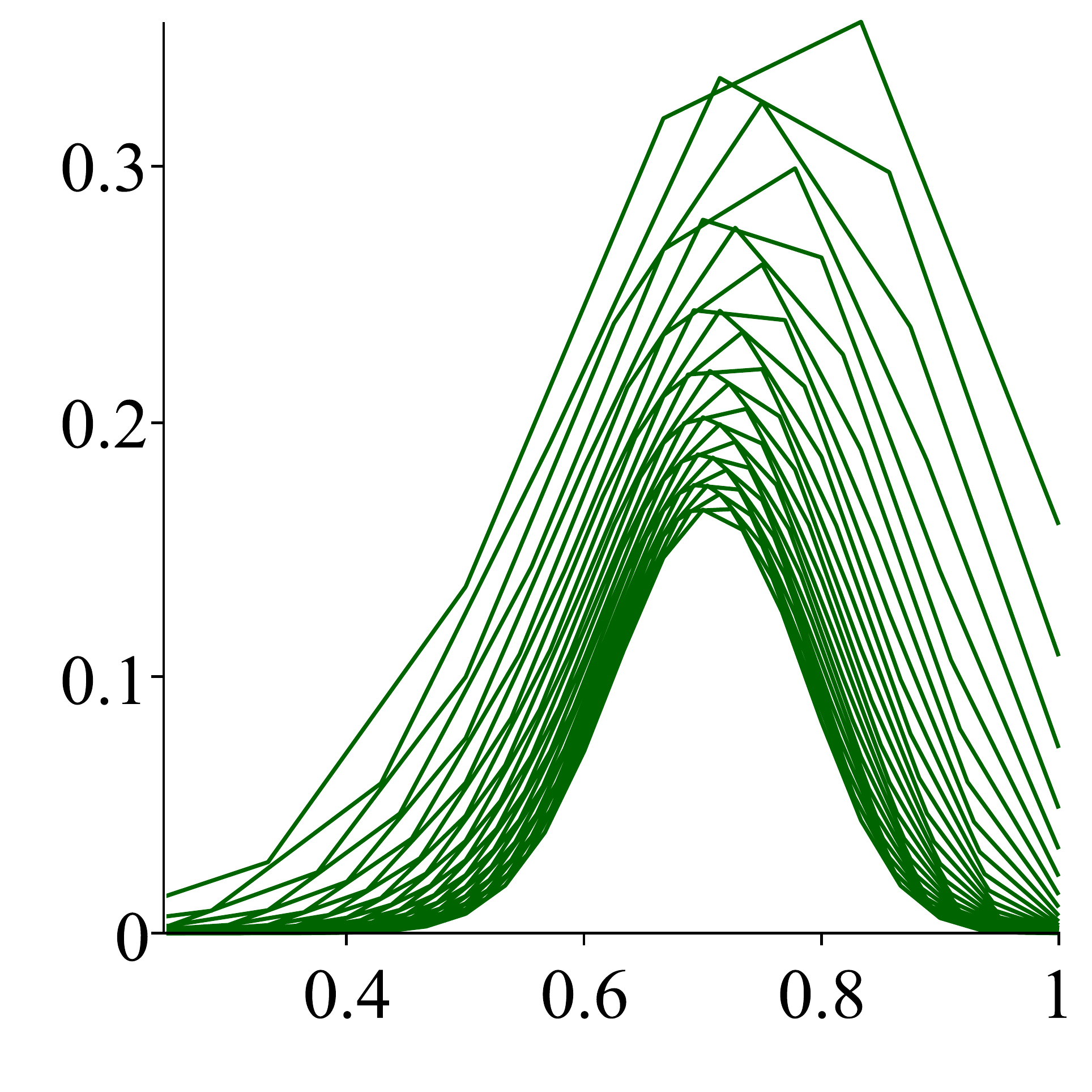} &
    \includegraphics[height=3cm]{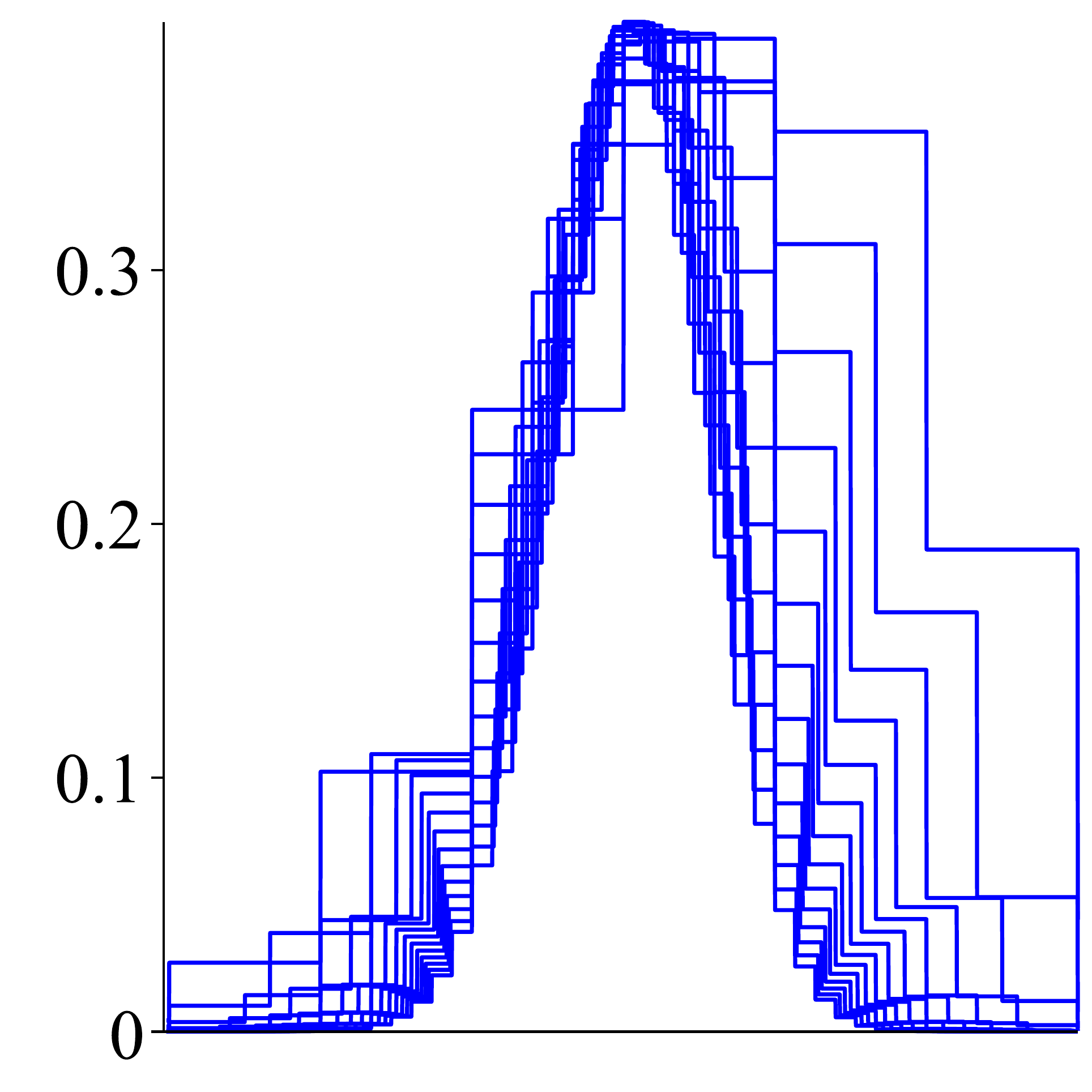} \\
    $\frac{|M_{n,x\log n}|}{\sum_{1\le j\le n}|M_{n,j}|}$ &
    $\frac{|M_{n,\tr{x\log n}}|}{\sum_{1\le j\le n}|M_{n,j}|}$ &
    $\frac{|M_{n,xn}|n^{xn}}{\sum_{1\le j\le n}|M_{n,j}|n^j}$ &
    $\frac{|M_{n,\tr{xn}}|n^{\tr{xn}}}
    {\sum_{1\le j\le n}|M_{n,j}|n^j}$ \\
    \multicolumn{2}{c}{\small $n=5\ell, 1\le \ell\le 12$} &
    \small $6\le n\le 30$ & \small $n=3\ell, 2\le \ell\le 15$
\end{tabular}
\end{center}
\caption{Different graphical renderings of histograms of
$|M_{n,k}|$ (the left two) and $|M_{n,k}|n^k$ (rightmost and
middle-right). All plots, except for the third one, are scaled by the standard deviations.}\label{F:histo}
\end{figure}

\begin{proof}
For the proof of Theorem~\ref{T:Matsunaga-numbers}, we begin with the
calculations of the mean and the variance. For convenience, define
\[
	\mathbb{E}\lpa{v^{X_n}} := \frac{P_n(v)}{P_n(1)}.
\]
Then, by \eqref{E:Pnv}, we obtain, for $n\ge4$
\[
    \mathbb{E}(X_n)
	= \frac{\displaystyle\sum_{0\le j\le n-2}
    (-1)^{j}\beta_{n-j} H_{n-j}}
	{\displaystyle\sum_{0\le j\le n-2}
    (-1)^{j}\beta_{n-j}}
	= \frac{\beta_nH_n-\beta_{n-1}H_{n-1}+-\cdots}
	{\beta_n-\beta_{n-1}+-\cdots},
\]
where $H_n^{[m]} := \sum_{1\le j\le n}j^{-m}$ denotes the harmonic
numbers and $H_n := H_n^{[1]}$. Similarly, the variance is given by
\[
	\mathbb{V}(X_n)
	= \frac{\displaystyle\sum_{0\le j\le n-2}
    (-1)^{j}\beta_{n-j} (H_{n-j}^2-H_{n-j}^{[2]})}
	{\displaystyle\sum_{0\le j\le n-2}(-1)^{j}\beta_{n-j}}
	-\mathbb{E}(X_n)^2+\mathbb{E}(X_n),
\]
for $n\ge4$. Note that each of these sums is itself an asymptotic
expansion in view of \eqref{E:en-ratio}; more precisely, if a given
sequence $\{\alpha_n\}$ satisfies $\alpha_{n-1}/\alpha_n\to c$, where
$c>0$, then, by \eqref{E:en-ratio}, we can group the terms in the
following way
\begin{align*}
    \frac{\displaystyle\sum_{0\le j\le n-2}
    (-1)^{j}\alpha_{n-j}\beta_{n-j}}
	{\displaystyle\sum_{0\le j\le n-2}
    (-1)^{j}\beta_{n-j}}
    &= \alpha_n \biggl(1-\frac{\beta_{n-1}}{\beta_n}
    \Lpa{1-\frac{\alpha_{n-1}}{\alpha_n}}\\
    &\qquad+\Lpa{\frac{\beta_{n-1}}{\beta_n}}^2
    \Lpa{1-\frac{\alpha_{n-1}}{\alpha_n}}
    -\frac{\beta_{n-2}}{\beta_n}
    \Lpa{1-\frac{\alpha_{n-2}}{\alpha_n}}+\cdots\biggr),
\end{align*}
where the terms decrease in powers of $n^{-j}W(n)^j$. Applying this
to the mean with $\alpha_n=H_n$, we then deduce, by \eqref{E:en}, that
\begin{align*}
	\mathbb{E}(X_n) &= H_n +\frac{W(n)}{n^2}
	+O\lpa{n^{-3}(\log n)^2},\\
	\mathbb{V}(X_n) &= H_n-H_n^{[2]}+\frac{W(n)}{n^2}
	+O\lpa{n^{-3}(\log n)^2},
\end{align*}
which is to be compared with the exact mean $H_n$ and exact variance
$H_n-H_n^{[2]}$ of the Stirling cycle distribution $|s_{n,k}|/n!$.
From known asymptotic expansions for the harmonic numbers, we also
have
\begin{align*}
	\mathbb{E}(X_n) &= \log n +\gamma + \frac1{2n}
	+\frac{12W(n)-1}{12n^2}+O\lpa{n^{-3}(\log n)^2},\\
	\mathbb{V}(X_n) &= \log n + \gamma - \frac{\pi^2}6
	+\frac3{2n} +\frac{12W(n)-7}{12n^2}
	+O\lpa{n^{-3}(\log n)^2},
\end{align*}
where $\gamma$ denotes the Euler–Mascheroni constant.

Similarly,
\[
	\mathbb{E}\lpa{v^{X_n}}
	= \binom{v+n-1}{n}
	\Lpa{1+\frac{W(n)(v-1)}{n^2}+O\lpa{n^{-3}(\log n)^2}},
\]
uniformly for $v=O(1)$. Then by singularity analysis
\cite{Flajolet2009}, we have
\[
    \mathbb{E}\lpa{v^{X_n}}
    = \frac{vn^{v-1}}{\Gamma(v+1)}\Lpa{1+O\lpa{n^{-1}(1+|v|^2)}},
\]
uniformly for $v=O(1)$, where the factor $1/\Gamma(v+1)$ is
interpreted as zero at negative integers; see \cite{Hwang1995}. Since
this has the standard form of the Quasi-powers framework (see
\cite{Flajolet2009,Hwang1998}), we then obtain the central limit
theorem $\mathscr{N}(\log n,\log n;(\log n)^{-\frac12})$.

On the other hand, by \eqref{E:Pnv} and \eqref{E:abs-Mnk}, we get
\[
	\mathbb{P}(X_n=k)
	= \frac{|M_{n,k}|}{P_n(1)}
	= \frac{|s_{n,k}|}{n!}\lpa{1+O\lpa{n^{-1}\log n}},
\]
uniformly for $1\le k\le n$. Thus the local limit theorem
\eqref{E:Mnk-llt} follows readily from the known results for
$|s_{n,k}|$ in \cite{Pavlov1989} or \cite{Hwang1994, Hwang1995}; see also the asymptotic approximation \eqref{E:mw1}.
\end{proof}
\section{Distribution of weighted Matsunaga numbers}

Since $B_n$ is, up to a minor shift and a normalizing factor, the sum
of $M_{n,k}n^k$ over all $k$, following the same spirit of
Proposition~\ref{P:en} and \eqref{E:Mnk-llt}, we examine more closely
how these weighted numbers $|M_{n,k}n^k|$ distribute over varying
$k$, which turns out to be very different from the unsigned Matsunaga
numbers; see Figure~\ref{F:histo} for a graphical illustration.

In the OEIS database, the sequence
\href{https://oeis.org/A056856}{A056856} equals indeed the
distribution of $|s_{n,k}|n^{k-1}$ for $1\le k\le n$, which is
mentioned to be related to rooted trees and unrooted planar trees,
together with several other formulae. See Section~\ref{S:three} for
two other sequences with the same asymptotic behaviors.

Our analysis shows that while $|M_{n,k}|$ are asymptotically normal
with logarithmic mean and logarithmic variance, their weighted
versions $|M_{n,k}|n^k$ also follow a normal limit law but with
linear mean and linear variance.
\begin{thm}\label{T:weighted-M}
\begin{align}\label{E:Yn}
    \{|M_{n,k}|n^k\}_k
    \simeq\LLT\lpa{\mu n, \sigma^2 n; n^{-\frac12}}
    \quad\text{with}\quad
    (\mu,\sigma^2) = \lpa{\log 2,\log2 - \tfrac12}.
\end{align}
\end{thm}

For convenience, define
\[
    \mathbb{E}\lpa{v^{Y_n}}
    := \frac{P_n(nv)}{P_n(n)}\qquad(n\ge2).
\]
Then, by \eqref{E:Pnv}, we have, for $n\ne3$
\begin{align}\label{E:Pnnv}
    \frac{P_n(nv)}{n!}
    = \sum_{0\le j<n}\binom{vn+n-j-1}{n-j}
    (-1)^{j}\beta_{n-j}.
\end{align}

\subsection{The total count}
The relation \eqref{E:Pnnv} implies that
\begin{equation}\label{E:Pnn}
\begin{split}
    \frac{P_n(n)}{n!}
    &= \frac1{n!}\sum_{1\le k\le n}|M_{n,k}|n^k
    = \sum_{0\le j<n}\binom{2n-1-j}{n-1}(-1)^j\beta_{n-j}\\
    &= \binom{2n-1}{n}\beta_n-\binom{2n-2}{n-1}\beta_{n-1}+-\cdots\\
    &= \frac{4^n\beta_n}{2\sqrt{\pi n}}
    \Lpa{1-\frac{4W(n)+1}{8n}+O\lpa{n^{-2}(\log n)^2}},
\end{split}
\end{equation}
in contrast to Matsunaga's identity \eqref{E:matsunaga}:
\[
    \frac1{n!}\sum_{1\le k\le n}M_{n,k}n^k = B_n-1.
\]
Their ratio is asymptotic to
\[
    \frac{\displaystyle\sum_{1\le k\le n}|M_{n,k}|n^k}
    {\displaystyle\sum_{1\le k\le n}M_{n,k}n^k}
    \sim \frac{W(n)}{2\sqrt{\pi}}\, n^{-\frac32}4^n.
\]
The sequence $P_n(n)$ begins for $n\ge1$ with $\{0, 6, 30, 3120,
135120, 11980080, 1231806240, \dots\}$, and
\[
    \left\{\frac{P_n(n)}{n!}\right\}_{n\ge1}
    = \{0, 3, 5, 130, 1126, 16639, 244406, 4107921, 74991344,
    1486313664,\dots\}.
\]

\subsection{Mean and variance}
Let $(\mu,\sigma^2) = \lpa{\log 2,\log2 - \tfrac12}$. From
\eqref{E:Pnnv}, we obtain
\[
    \sum_{1\le k\le n}k|M_{n,k}|n^k
    =n\cdot n!\sum_{0\le j<n}\binom{2n-1-j}{n-j}
	(-1)^j\beta_{n-j}\lpa{H_{2n-j-1}-H_{n-1}},
\]
for $n\ne3$. Then
\begin{align*}
	\mathbb{E}(Y_n)
    &= \frac1{P_n(n)}\sum_{1\le k\le n}k|M_{n,k}|n^k
    = \mu n +\frac14 -\frac{4W(n)-1}{16n}
    +O\lpa{n^{-2}(\log n)^2}.
\end{align*}	
Similarly,
\[
    \mathbb{V}(Y_n)
    = \sigma^2 n -\frac18 -\frac1{12n}+O\lpa{n^{-2}(\log n)^2}.
\]

\subsection{Asymptotics of $|s_{n,k}|$}
To prove the local limit theorem of $Y_n$, we need finer asymptotic
approximations to the Stirling numbers of the first kind; see
\cite{Bender1973,Moser1958,Temme1993} for more information.
\begin{itemize}
    \item For $1\le k=O(\log n)$ (see \cite{Hwang1995}),
    \begin{align}\label{E:mw1}
        \frac{|s_{n,k}|}{n!}
        &= \frac{(\log n)^{k-1}}{n\Gamma(1+\frac{k-1}{\log n})
        (k-1)!}\Lpa{1+O\lpa{k(\log n)^{-2}}}.
    \end{align}
    \item For $k\to\infty$ and $n-k\to\infty$ (see
    \cite{Bender1973,Moser1958}),
    \begin{align}\label{E:mw2}
        |s_{n,k}|
        = \frac{r^{-k}\Gamma(n+r)}
        {\sqrt{2\pi V}\,\Gamma(r)}
        \Lpa{1+O\lpa{V^{-1}}},
    \end{align}
    where $r=r_{n,k}>0$ solves the saddle-point equation $r(\psi(n+r)-\psi(r))=k$
    and $V := k+r^2(\psi'(n+r)-\psi'(r))$.
    Here $\psi(t)$ denotes the digamma function (derivative of $\log\Gamma(t)$).
    \item For $0\le \ell = n-k =o(\sqrt{n})$ (see \cite{Moser1958}),
    \begin{align*}
        |s_{n,k}| = \frac{n^{2\ell}}{\ell! 2^\ell}
        \Lpa{1+O\lpa{(\ell+1)^2 n^{-1}}}.
    \end{align*}
\end{itemize}

\subsection{Local limit theorem}
We prove Theorem~\ref{T:weighted-M}.
\begin{figure}[!ht]
    \begin{center}
        \begin{tabular}{cc}
            \includegraphics[height=3cm]{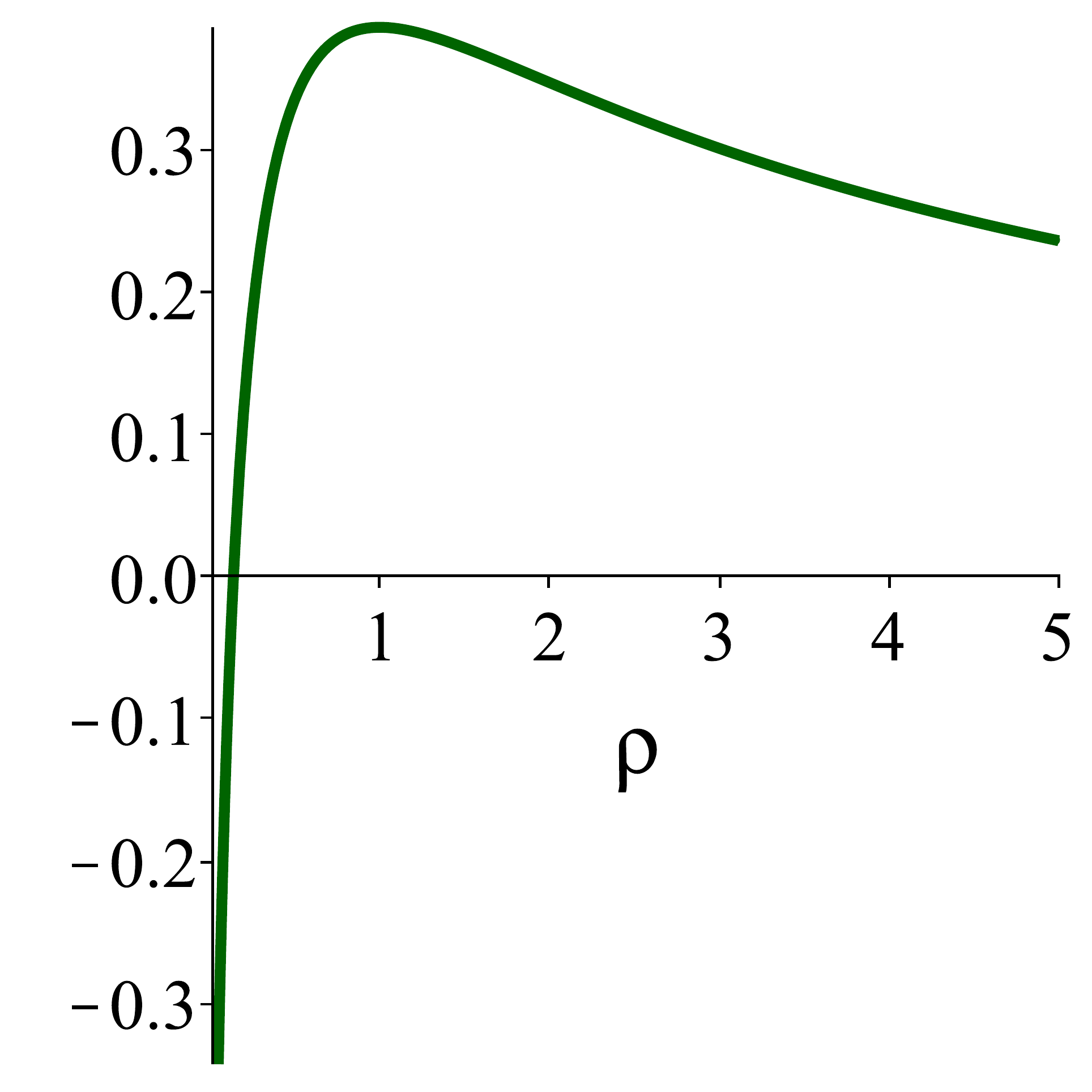}
            & \includegraphics[height=3cm]{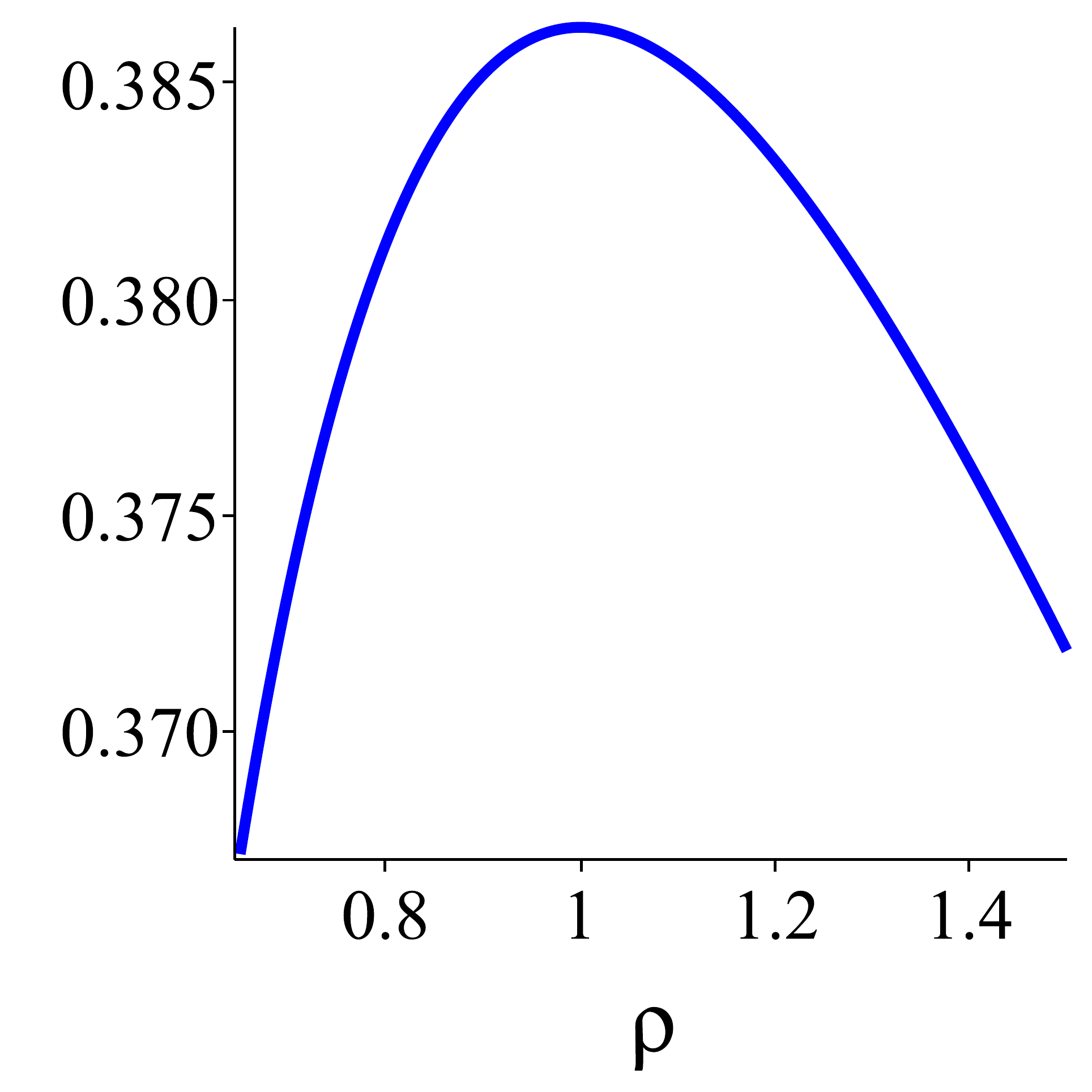}\\
            $\varphi(\rho)$, $\rho\in(0,5)$ &
            $\varphi(\rho)$, $\rho\in(0.65,1.5)$
        \end{tabular}
    \end{center}
    \caption{The behaviors of $\varphi(\rho)$.} \label{F:varphi}
\end{figure}
We first identify $k$ at which $|M_{n,k}|n^k$ reaches the maximum
value for fixed $n$. We substitute first $k\sim\tau n$, $\tau\in[0,1]$,
and $r\sim\rho n$, $\rho>0$, in the saddle-point approximation
\eqref{E:mw2}, and obtain, by using \eqref{E:bnk-ae} and Stirling's
formula:
\begin{align*}
    \log\frac{|M_{n,k}|n^k}{\beta_n}
    &\sim \log(|s_{n,k}|n^k)
    \sim \lpa{\log n +\varphi(\rho)}n,
\end{align*}
where $\varphi(\rho) := \rho(1-\log\rho)\log(1+\rho^{-1})
+\log(1+\rho)-1$, and the pair $(\rho, \tau)$ is connected by the
relation $\rho\log(1+\rho^{-1})=\tau$. A simple  calculus shows that
the image of $\varphi(\rho)$ (see Figure~\ref{F:varphi}) lies in the range $(-\infty,\varphi(1)]$
for $\rho\in(0,\infty)$, where the maximum of $\varphi$ at $\rho=1$
equals $\varphi(1)=2\log 2-1\approx 0.38629\dots$. Then $\rho=1$ implies that
$k\sim \mu n$, where $\mu=\log 2$.

Once the peak of $|M_{n,k}|n^k$ is identified to be at $k\sim \mu n$,
we then refine all the asymptotic expansions by writing $k=\mu n +
x\sigma \sqrt{n}$, and then solving the saddle-point equation we obtain
\[
    r = n+\frac{x\sqrt{n}}{\sigma}
    +\frac{x^2-2\sigma^2}{8\sigma^2}
    +\frac{x}{96\sigma^7\sqrt{n}}
    \lpa{12\sigma^4-8\sigma^2x^2-6\sigma^2+3x^2}
    +O\Lpa{\frac{1+x^4}{n}},
\]
whenever $x=o(n^{\frac16})$. Substituting this expansion into
Moser and Wyman's saddle-point approximation \eqref{E:mw2} and using
\eqref{E:Pnn}, gives the local Gaussian behavior of $|M_{n,k}|n^k$:
\begin{equation}\label{E:Xn}
\begin{split}
    \mathbb{P}(Y_n=k)
    &= \frac{|M_{n,k}|n^k}{P_n(n)}\\
    &= \frac{e^{-\frac12x^2}}{\sqrt{2\pi \sigma^2n}}
    \Lpa{1+\frac{x((4\sigma^2-1)x^2-3(2\sigma^2-1))}
    {24\sigma^3\sqrt{n}}+O\Lpa{\frac{(1+x^4)\log n}{n}}},
\end{split}
\end{equation}
uniformly for $k=\mu n + x\sigma\sqrt{n}$ with $x=o(n^{\frac16})$.

When $|x|$ lies outside the central range, say $|x|\ge n^{\frac18}$
(for simplicity), we use the crude bounds (with $k_\pm :=
\mu n \pm \sigma n^{\frac58}$)
\begin{align*}
    \mathbb{P}\lpa{Y_n\le k_-}
    +\mathbb{P}\lpa{Y_n\ge k_+}
    &= O\llpa{n \max_{|k-\mu n|\ge \sigma n^{\frac58}}
    \frac{|s_{n,k}|n^k}{n!}}
    = O\Lpa{\frac{n|s_{n,k_-}|n^{k_-}}{n!}}
    = O\lpa{n^{\frac12} e^{-\frac12 n^{\frac14}}},
\end{align*}
by \eqref{E:Xn}. This together with \eqref{E:Xn} completes the proof
of Theorem~\ref{T:weighted-M}.

\subsection{Three sequences with the same $\mathscr{N}(\mu n,\sigma^2n)$ asymptotic distribution}
\label{S:three}

Our proof of Theorem~\ref{T:weighted-M} also implies the same local
limit theorem \eqref{E:Yn} for the sequence $|s_{n,k}|n^{k-1}$
(\href{https://oeis.org/A056856}{A056856}).

Two other OEIS sequences with the same asymptotic behaviors are
\begin{itemize}
	\item \href{https://oeis.org/A220883}{A220883}:
	\[
	    [z^k]\prod_{1\le j<n}(j+(n+1)z)
		= |s_{n,k+1}|(n+1)^k.
	\]
	
	\item \href{https://oeis.org/A260887}{A260887}:
	\[
	    [z^k]\prod_{2\le j\le n}(j+nz)
		= n^k\sum_{0\le j\le k}|s_{n+1,j+1}|(-1)^{k-j}.
	\]
\end{itemize}
The proof of Theorem~\ref{T:weighted-M} also extends to these cases;
for example, by Moser and Wyman's saddle-point analysis, we first have
\[
	[z^k]\prod_{2\le j\le n}(j+nz)
	= \frac{r^{-k}\Gamma(n+1+r)}{\sqrt{2\pi V}\,\Gamma(r+2)}
	\lpa{1+O\lpa{V^{-1}}},
\]
where $r>0$ solves the equation $n(\psi(n+r+1)-\psi(r+2))=k$ and $V
:= k+r^2(\psi'(n+r+1)-\psi'(r+2))$. Then we follow the same proof of
Theorem~\ref{T:weighted-M} and obtain, for
\href{https://oeis.org/A260887}{A260887},
\[
	\left\{[z^k]\prod_{2\le j\le n}(j+nz)\right\}_{k}
	\simeq\LLT\lpa{\mu n,\sigma^2 n, n^{-\frac12}},
\]
The same result holds for \href{https://oeis.org/A220883}{A220883}.
In contrast, the neighboring sequence
\href{https://oeis.org/A220884}{A220884}
\[
	[z^k]\prod_{2\le j\le n}(j+(n+1-j)z)
\]
satisfies $\mathscr{N}(\frac12n, \frac16n; n^{-\frac12})$, which can
be proved either by Harper's real-rootedness approach or the
classical characteristic function approach (using L\'evy's continuity
theorem); see \cite[p.~108]{Hwang2020}.
\section{Distribution of Arima numbers}
\label{S:arima}

\subsection{Yoriyuki Arima and his 1763 book on Bell numbers}

Yoriyuki Arima (1714--1783) was born in Kurume Domain and then became
the Feudal lord there at the age of 16. As was common at that time,
he also used several different names during his life time. He
apprenticed himself to Nushizumu Yamaji ({\footnotesize{山路主住}}),
and later wrote over 40 books during 1745--1766. He selected and
compiled 150 typical questions from these books and published the
solutions in the compendium book \emph{Shuki Sanpo} \cite{Arima1769}
({\footnotesize{拾璣算法}}) in five volumes under the pen-name Bunkei
Toyoda ({\footnotesize{豐田文景}}). This influential book was regarded highly in Wasan at that time and played a significant role in
popularizing the theory and techniques developed in the Seki School.
Not only the materials are well-organized, but the style is
comprehensible, which is unique and was thought to be a valuable
contribution to the developments of Wasan in and after the Edo period.
For more information on Arima's life and mathematical works, see
\cite{Fujiwara1957}.

Our next focus in this paper lies on his 1763 book \cite{Arima1763},
which is devoted to two different procedures of computing Bell
numbers, a summary of which being given as follows.

\begin{itemize}

\item Matsunaga's procedure \eqref{E:matsunaga} is first examined
(pp.~\href{https://kotenseki.nijl.ac.jp/biblio/100247598/viewer/3}{3--5}), and the values of $B_n$ for $n=2,\dots, 13$ are computed.

\item Values of the Matsunaga numbers $M_{n,k}$ are listed on page
\href{https://kotenseki.nijl.ac.jp/biblio/100247598/viewer/6}{6}.

\item Computing $B_n$ through the better recurrence
\eqref{E:bell-formulae} (pp.~\href{https://kotenseki.nijl.ac.jp/biblio/100247598/viewer/7}{7--12}), which is simplified with the
additional (common) tabular trick (trading off space for computing time) that
\[
	b_{n,k} := \binom{n-1}{k-1}B_{n-k}
	= \begin{cases}
		\displaystyle
		\sum_{1\le j<n} b_{n-1,j},
		&\text{if }k=1;\\
		\displaystyle
		\frac{n-1}{k-1}\,b_{n-1,k-1},
		&\text{if }2\le k\le n.
	\end{cases}
\]
It follows that $B_n = \sum_{1\le k\le n}b_{n,k}$.

\item Stirling numbers of the first kind $s_{n,k}$ are tabulated for
$2\le n\le 10$ (pp.~\href{https://kotenseki.nijl.ac.jp/biblio/100247598/viewer/12}{12--16}) via the expansion of
$x(x-1)\cdots(x-n+1)$.

\item Then the next twenty pages or so (pp.~\href{https://kotenseki.nijl.ac.jp/biblio/100247598/viewer/17}{17--36}) give a detailed
inductive discussion to compute the number of arrangements when there
are $k_1$ blocks of size $1$, $k_2$ blocks of size $2$, etc.
(essentially the coefficients of the Bell polynomials):
\[
	B_n(k_1,k_2,\dots) = \frac{n!}{1!^{k_1}2!^{k_2}\cdots
	k_1!k_2!\cdots}.
\]

\item Bell numbers are computed (pp.~\href{https://kotenseki.nijl.ac.jp/biblio/100247598/viewer/36}{36--41}) by collecting all
different block configurations (or adding the coefficients in the
Bell polynomials).

\item Matsunaga numbers $M_{n,k}$ are computed on pages \href{https://kotenseki.nijl.ac.jp/biblio/100247598/viewer/42}{42--47}, where
$\beta_n$ is given on page \href{https://kotenseki.nijl.ac.jp/biblio/100247598/viewer/43}{43}.

\item The remaining pages (\href{https://kotenseki.nijl.ac.jp/biblio/100247598/viewer/48}{48--55}) discuss multinomial coefficients.

\end{itemize}

\subsection{Arima numbers}
In this section, for completeness, we prove the LLT of the Arima
numbers
\[
    A_{n,k}:= \binom{n}{k}B_{n-k},
\]
which appeared in Arima's 1763 book (pages \href{https://kotenseki.nijl.ac.jp/biblio/100247598/viewer/7}{7--8} in file order) and also
sequence \href{https://oeis.org/A056857}{A056857} in the OEIS; its
row-reversed version corresponds to
\href{https://oeis.org/A056860}{A056860}. Since our main interest
lies in the asymptotic distribution, we consider
$\binom{n}{k}B_{n-k}$ instead of the original
$\binom{n-1}{k}B_{n-1-k}$.

Another closely related sequence is
\href{https://oeis.org/A175757}{A175757} (number of blocks of a given
size in set partitions), which is the same as A056857 but without the
leftmost column.

Several interpretations or contexts where these sequences arise can
be found on the OEIS page; see also \cite[p.~178]{Mansour2013} for
the connection to weak records in set partitions. For example, it
gives the size of the block containing $1$, as well as the number of
successive equalities in set partitions.

On the other hand, the sequence
\href{https://oeis.org/A005578}{A005578} ($\cl{\frac132^n}$) is
sometimes referred to as the Arima sequence.

\begin{table}[!ht]
\begin{center}
\begin{tabular}{c|rrrrrrrr|c}
	$n\backslash k$
	& $0$ & $1$ & $2$ & $3$ & $4$ & $5$ & $6$ & $7$
    & $B_{n+1}$ \\ \hline
    $1$ & $1$ & $1$ &&&&&&& $2$ \\
    $2$ & $2$ & $2$ & $1$ &&&&&& $5$ \\
    $3$ & $5$ & $6$ & $3$ & $1$ &&&&& $15$\\
	$4$ & $15$ & $20$ & $12$ & $4$ & $1$ &&&& $52$\\
    $5$ & $52$ & $75$ & $50$ & $20$ & $5$ & $1$ &&& $203$ \\
    $6$ & $203$ & $312$ & $225$ & $100$
    & $30$ & $6$ & $1$ && $877$ \\
    $7$ & $877$ & $1421$ & $1092$ & $525$ & $175$ & $42$ & $7$ & $1$
    & $4140$
\end{tabular}
\end{center}
\caption{The values of $A_{n,k}$ for $n=1,\dots,7$ and
$1\le k\le n$ as already given on
\cite[Page 8]{Arima1763}.}\label{T:arima}
\end{table}

\begin{thm}
\begin{align}\label{E:arima-llt}
    \left\{\binom{n}{k}B_{n-k}\right\}_k
    \simeqq \LLT\lpa{\log n,\log n; (\log n)^{-\frac12}}.
\end{align}
\end{thm}
The same LLT holds for the sequence
\href{https://oeis.org/A175757}{A175757}, and for
\href{https://oeis.org/A056860}{A056860}, we have
\[
    \left\{\binom{n}{k}B_{k}\right\}_k
    \simeqq \LLT\lpa{n-\log n,\log n; (\log n)^{-\frac12}}.
\]
\begin{proof}
Recall first the known saddle-point approximation to Bell numbers
(see \cite{deBruijn1981,Moser1955})
\begin{equation}\label{E:Bn1}
\begin{split}
    B_n &= n![z^n]e^{e^z-1}\\
	&=\frac{e^{(w+\frac1{w}-1)n-1}}
    {\sqrt{w+1}}\Lpa{1
	-\frac{w^2(2w^2+7w+10)}{24n(w+1)^3}
	+O\lpa{n^{-2}(\log n)^2}}.
\end{split}	
\end{equation}
where $w=W(n)$.
From this and the asymptotic expansion \eqref{E:Wn1}, we can quickly
see why \eqref{E:arima-llt} holds. Since $\binom{n}{k}<2^n$ for all
$k$, and $\log B_n\sim nw\sim n\log n -n\log\log n$, meaning that
$B_n$ is very close to factorial, and thus larger than the binomial
term. The largest term of $B_{n-k}$ is when $k=0$. More precisely,
when $k$ is small,
\begin{align}\label{E:arima-heuristic}
    \binom{n}{k}\frac{B_{n-k}}{B_{n}}
    \sim \frac{n^k}{k!} \times n^{-k} (\log n)^{k}
    = \frac{(\log n)^{k}}{k!},
\end{align}
and thus a CLT with logarithmic mean and logarithmic variance is
naturally expected.

For a rigorous proof, we begin with the calculation of the mean.
First,
\[
    A_{n,k} = n![z^nv^k]
    e^{e^z-1+vz},
\]
and from this exponential generating function, we can derive the
(exact) mean and the variance to be
\[
    \mu_n := \frac{nB_n}{B_{n+1}}, \quad\text{and}
    \quad \sigma_n^2 := \frac{n(n-1)B_{n-1}+nB_n}{B_{n+1}}
    -\Lpa{\frac{nB_n}{B_{n+1}}}^2,
\]
respectively. The asymptotic mean and variance then follow from
\eqref{E:Bn1} and the asymptotic expansion \eqref{E:Wn1}; indeed,
finer expansions give
\begin{align*}
    \mu_n &= w\Lpa{1-\frac{w^2}{2n(w+1)^2}
	+O\lpa{n^{-2}(\log n)^2}},\\
	\sigma_n^2 &= w\Lpa{1-\frac{w(3w+2)}{2n(w+1)^2}
	+O\lpa{n^{-2}(\log n)^2}}.
\end{align*}

An alternative approach by applying saddle-point method and
Quasi-powers theorem \cite{Hwang1998, Flajolet2009} is as follows.
First, we derive, again by saddle-point method, the asymptotic
approximation
\begin{align*}
    n![z^n]e^{e^z-1+vz}
    &= \frac{n!}{2\pi i}\oint_{|z|=r}
    z^{-n-1} e^{e^z-1+vz}\dd z\nonumber \\
    &= \frac{n!r^{-n}e^{\frac nr-1+vr}}
    {\sqrt{2\pi((r+1)n-vr^2)}}
    \lpa{1+O\lpa{rn^{-1}}},
\end{align*}
uniformly for complex $v$ in a small neighborhood of unity
$|v-1|=o(1)$, where $r=r(v)>0$ solves the equation $re^r+vr=n$. For
large $n$, a direct bootstrapping argument gives the asymptotic
expansion
\[
    r = w-\frac{vw^2}{n(w+1)}
    -\frac{v^2w^3(w^2-2)}{2n^2(w+1)^3}+O\lpa{n^{-3}(\log n)^3},
\]
uniformly for bounded $v$. From these expansions, we then obtain
\[
    \frac{n![z^n]e^{e^z-1+vz} }{n![z^n]e^{e^z-1+z}}
    = e^{w(v-1)}\lpa{1+O\lpa{n^{-1}\log n}},
\]
uniformly for $|v-1|=o(1)$. This implies an asymptotic
Poisson($W(n)$) distribution for the underlying random variables,
and, in particular, the CLT $\mathscr{N}(\log n,\log n)$ follows.
The stronger LLT \eqref{E:arima-llt} is proved by \eqref{E:Bn1} and
standard approximations for binomial coefficients, following the same
procedure used in the proof of \eqref{E:Yn}; in particular, we use
the expansion of $\binom{n}{k}$:
\[
    \binom{n}{k}
    = \frac{\tau^{-\tau n}(1-\tau)^{-(1-\tau)n}}
    {\sqrt{2\pi \tau(1-\tau)n}}
    \Lpa{1+O\lpa{k^{-1}+(n-k)^{-1}}},
\]
holds uniformly as $k, n-k\to\infty$, where $\tau = k/n$. Outside
this range, $\binom{n}{k}B_{n-k}$ is asymptotically negligible.
\end{proof}

\begin{figure}[!ht]
\begin{center}
\begin{tabular}{cccc}
    \includegraphics[height=3cm]{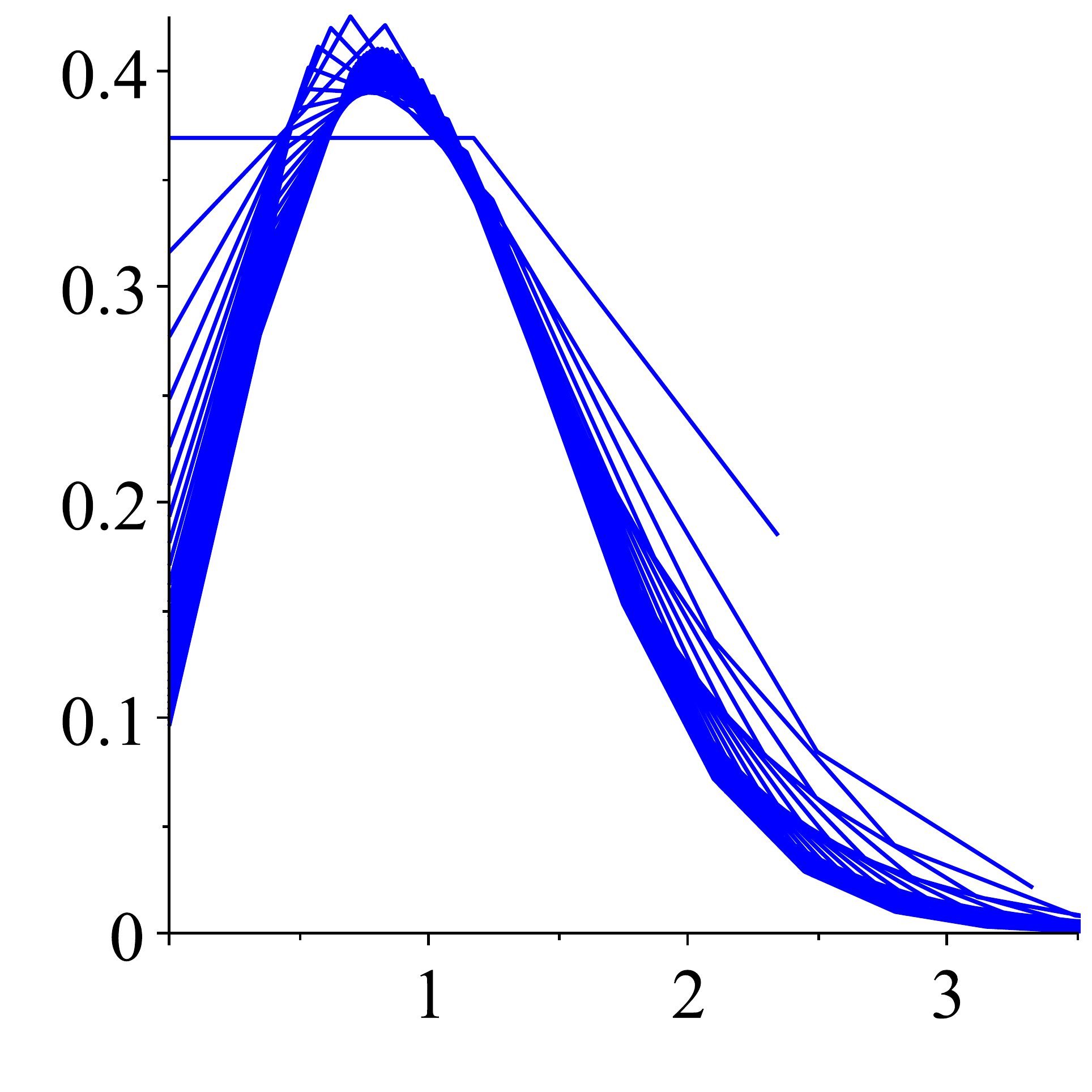} &
    \includegraphics[height=3cm]{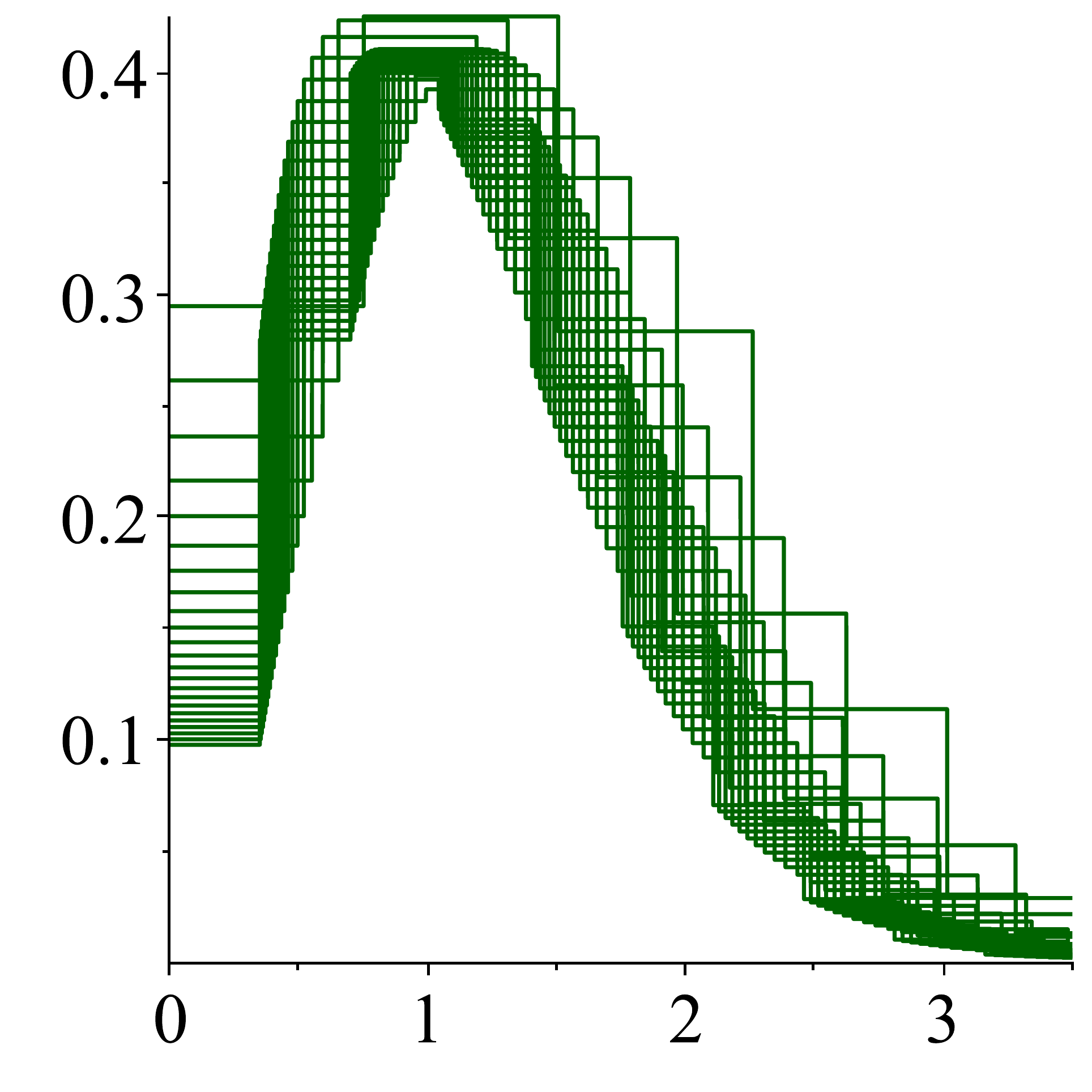} &
    \includegraphics[height=3cm]{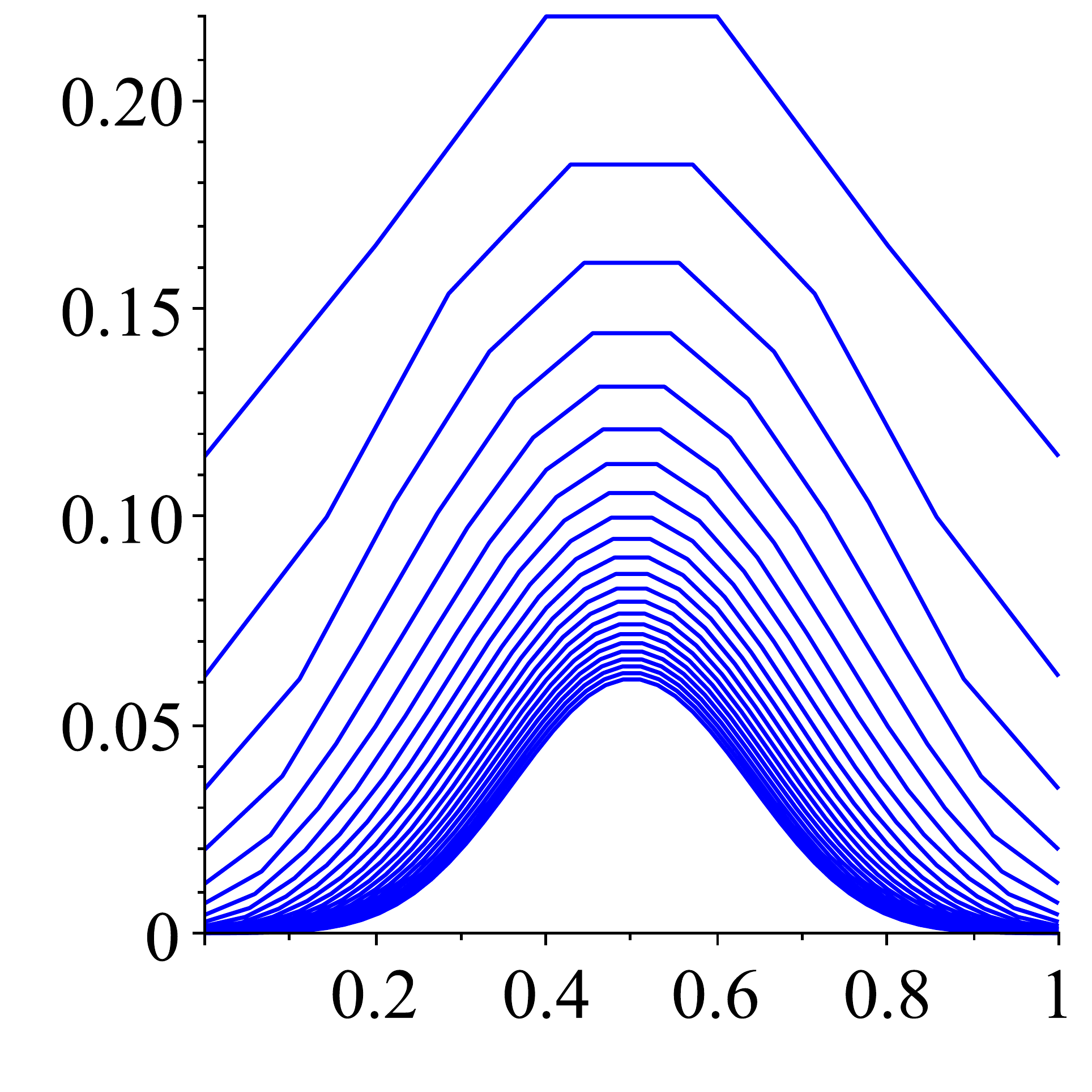} &
    \includegraphics[height=3cm]{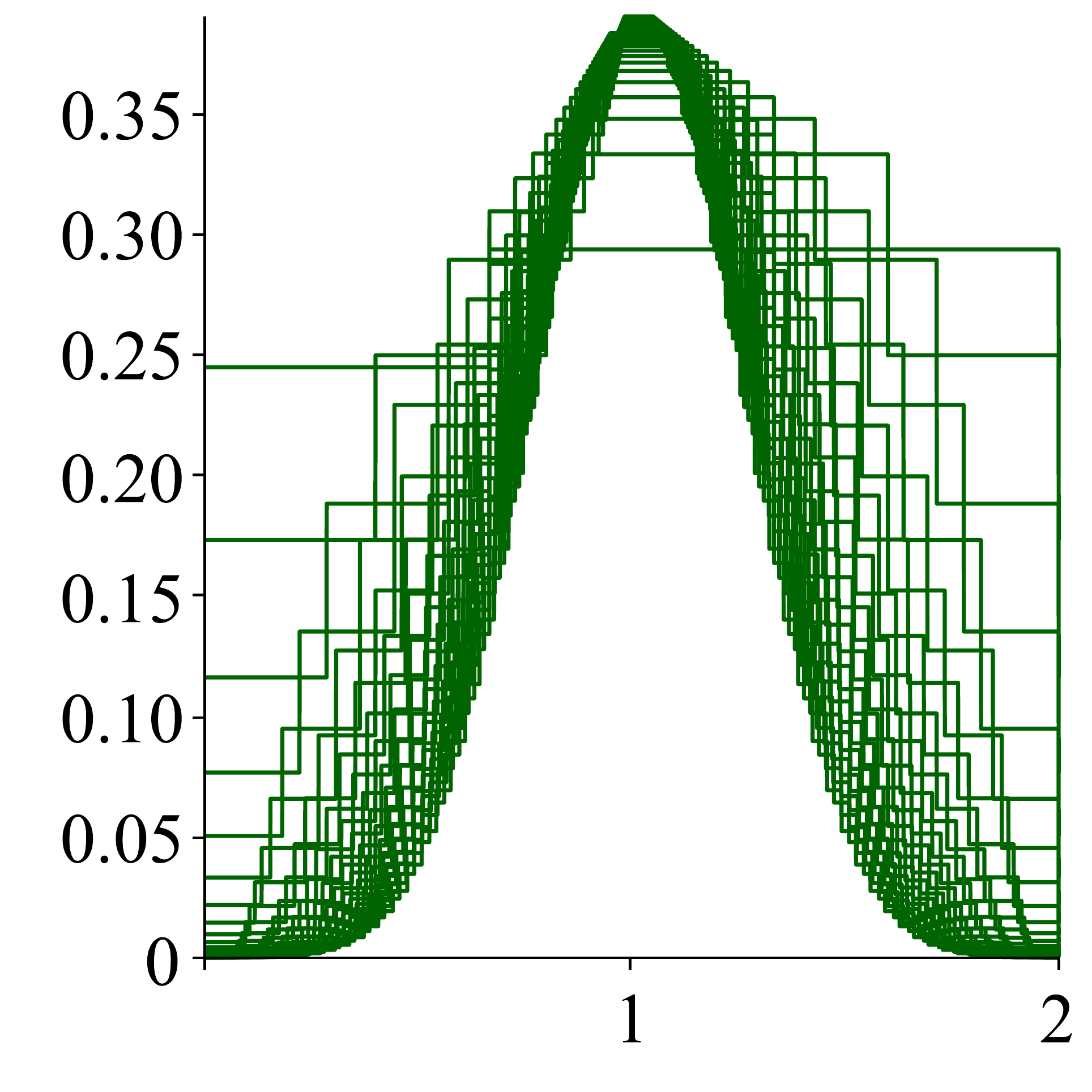}\\
    $\frac{\binom{n}{x\log n}B_{n-x\log n}}{\sum_{1\le j\le n}\binom{n}{j}B_{n-j}}$ &
    $\frac{\binom{n}{\tr{x\log n}}B_{n-\tr{x\log n}}}{\sum_{1\le j\le n}\binom{n}{j}B_{n-j}}$ &
    $\frac{\binom{n}{xn}B_{xn}B_{n-x n}}{\sum_{1\le j\le n}\binom{n}{j}B_j B_{n-j}}$ &
    $\frac{\binom{n}{\lfloor{\frac12xn\rfloor}}B_{\lfloor{\frac12xn\rfloor}}B_{n-\lfloor{\frac12xn\rfloor}}}{\sum_{1\le j\le n}\binom{n}{j}B_j B_{n-j}}$  \\
    \small $n=2\ell, 1\le\ell\le 25$ & \multicolumn{2}{c}{\small $n=2\ell+1, 2\le \ell\le 24$} &\small $n=2\ell+1, 1\le \ell\le 29$
\end{tabular}
\end{center}
\caption{Histograms of Arima numbers $\binom{n}{k}B_{n-k}$ (the left two plots) and
the sequence $\binom{n}{k}B_kB_{n-k}$ (\href{https://oeis.org/A033306}{A033306}; the right two
plots). All plots, except for the third one, are scaled by the standard deviations.}
\end{figure}

\subsection{Asymptotic normality of a few variants}
A few other distributions with the same logarithmic type LLT are
listed in Table~\ref{T:log-oeis}, the proofs being completely similar
and omitted.

\def\arraystretch{1.5}
\begin{table}[!ht]
\begin{center}
    \begin{tabular}{c|ccccc}
        OEIS & 
        \href{https://oeis.org/A078937}{A078937} &
        \href{https://oeis.org/A078938}{A078938} &
        \href{https://oeis.org/A078939}{A078939} &
	    \href{https://oeis.org/A124323}{A124323} & 
	    \href{https://oeis.org/A086659}{A086659} \\ \hline 
        EGF & $e^{2(e^z-1)+vz}$ & 
        $e^{3(e^z-1)+vz}$ &
        $e^{4(e^z-1)+vz}$ &
	    $e^{e^z-1+(v-1)z}$ &
	    $e^{e^z-1+(v-1)z}-e^{vz}$
    \end{tabular}
\end{center}
\caption{A few OEIS sequences (together with their EGFs) with the 
same $\LLT\lpa{\log n,\log n;(\log n)^{-\frac12}}$ asymptotic behavior
as Arima numbers.} \label{T:log-oeis}
\end{table}

In particular, \href{https://oeis.org/A124323}{A124323} enumerates
singletons in set partitions.

These EGFs are reminiscent of $r$-Bell numbers with the EGF
$e^{v(e^z-1)+rz}$ \cite{Broder1984, Mezo2011}; see also
\cite{Hennessy2011} for exponential Riordan arrays of the type
$e^{(\alpha+v)(e^z-1)+(\alpha-\beta)z}$. These types of EGFs lead
however to normal limit laws with mean of order $\frac{n}{W(\frac
n{\alpha+1})}$ and variance of order $\frac{n}{W(\frac n{\alpha+1})}$
when $\alpha>0$ and $\beta\le 2\alpha$, and of order
$\frac{n}{W(n)^2}$ when $\alpha=0$ and $\beta\le0$,
similar to that of $\Stirling{n}{k}$.

\subsection{A less expected $\LLT(\frac12n,
\frac14n\log n)$ for $\binom{n}{k}B_kB_{n-k}$}
On the other hand, from the intuitive reasoning given in
\eqref{E:arima-heuristic} (which can be made rigorous), it is of
interest to scrutinize the corresponding balanced version
$\binom{n}{k}B_kB_{n-k}$, which is sequence \href{https://oeis.org/A033306}{A033306} with the EGF
$e^{e^{vz}+e^z-2}$. In this case, the peak of the distribution is
reached at $k=\tr{\frac12n}$ and $k=\cl{\frac12n}$. Indeed, the mean
is identically $\frac12n$, and the variance can be computed by
\[
	\frac n4 +\frac{n(n-1)\tilde{B}_{n-1}}{4\tilde{B}_n},
\]
where $\tilde{B}_n := n![z^n]e^{2(e^z-1)}$ is sequence
\href{https://oeis.org/A001861}{A001861} in the
OEIS (or the $n$th moment of a Poisson distribution with mean $2$).
By the asymptotic expansion ($\tilde{w}:= W(\frac12n)$)
\begin{align}\label{E:bn}
	\tilde{B}_n = \frac{e^{(\tilde{w}-1+\log 2+\tilde{w}^{-1})n-2}}
	{\sqrt{\tilde{w}+1}}\Lpa{1-\frac{\tilde{w}^2(2\tilde{w}^2+7\tilde{w}+10)}
	{24(\tilde{w}+1)^3n}+O\lpa{n^{-2}(\log n)^2}},
\end{align}
we have that the variance satisfies
\begin{align*}
	\MoveEqLeft
	\frac n4 +\frac{n(n-1)\tilde{B}_{n-1}}{4\tilde{B}_n}
	= \frac{\tilde{w}+1}4\,n-\frac{\tilde{w}(\tilde{w}^2+2\tilde{w}+2)}{8(\tilde{w}+1)^2}
	+O\lpa{n^{-1}(\log n)^2},
\end{align*}
showing the less expected $n\log n$ asymptotic behavior. By
\eqref{E:bn} and Stirling's formula, we obtain, for
$k=\frac12n+\frac12x\sqrt{n(\tilde{w}+1)}$,
\[
	\frac{\binom{n}{k}B_kB_{n-k}}{\tilde{B}_n}
	=\frac{e^{-\frac12x^2}}{\sqrt{\frac12\pi n(\tilde{w}+1)}}
	\Lpa{1+O\Lpa{\frac{\log n}{n}(1+x^4)}},
\]
uniformly in the range $x=o(n^{\frac14}(\log n)^{\frac14})$
(wider than the usual range $o(n^{\frac16})$ due to symmetry of the
distribution). Smallness of the distribution outside this range also
follows from similar arguments, and we deduce that
\[
	\binom{n}{k}B_kB_{n-k}
	\simeq \LLT\Lpa{\frac n2, \frac{n\tilde{w}}{4};
	\frac{\log n}{n}}.
\]

Finally, sequence \href{https://oeis.org/A174640}{A174640}, which
equals $\binom{n}{k}B_kB_{n-k}-B_n+1$ also follows asymptotically the
same behaviors.
\section{Conclusions}

We clarified the two Edo-period procedures in the eighteenth century
in Japan to compute Bell numbers, and derived fine asymptotic
and distributional properties of several classes of numbers arising
in such procedures, shedding new light on the early history and
developments of Bell and related numbers.

In addition to the modern perspectives given in this paper, our study
of these old materials also suggests several other questions. For
example, the change of the asymptotic distributions from $\LLT(\log
n, \log n)$ for $|M_{n,k}|$ to $\LLT(\mu n, \sigma^2n)$ for
$|M_{n,k}|n^k$ suggests examining other weighted cases, say
$\binom{n}{k}|M_{n,k}|$ or more generally $\pi_{n,k}|M_{n,k}|$ for a given
sequence $\pi_{n,k}$, and studying the
corresponding phase transitions of limit laws. This and several
related questions will be explored in detail elsewhere.
\section*{Acknowledgements}

We thank Guan-Huei Duh and Jin-Wen Chen for their assistance in the
long collection process of the literature, as well as the
clarification of the history of Chinese and Japanese Mathematics.
Special thanks go to Yu-Sheng Chang who provided the nice tikz-code
for the Genjik\={o}nozus in Figure~\ref{F:Genjimon1}.

\bibliographystyle{abbrv}
\bibliography{bell-numbers}

\end{CJK}
\end{document}